\documentclass{scrartcl}

\usepackage[
theoremdefs,
final,
]{allergy}
\usepackage{unixode}

\usepackage[font=small]{caption}

\usepackage{helvet}

\usepackage{booktabs}
\usepackage{subfig}

\usetikzlibrary{calc}
\usetikzlibrary{spy}
\usepackage{pgfplots}
\newcommand{\dada}[1]{#1}
\newcommand{\olol}[1]{#1}

\newcommand{\R}{\mathbb R}
\newcommand{\Z}{\mathbb Z}

\newcommand{\e}{\ensuremath{\mathrm{e}}}

\newcommand{\ii}{\ensuremath{\mathrm{i}}}

\newcommand*\dd{\mathrm{d}}
\newcommand{\GG}{\nabla g}

\newcommand*\CC{\mathbb{C}}
\newcommand*\CP{\CC P}
\newcommand*\uu[1]{S(#1)}
\newcommand*\Tr{\mathrm{Tr}}

\usepackage{relsize}
\usepackage{xspace}
\newcommand*\Shake{{\smaller SHAKE}\xspace}
\newcommand*\Rattle{{\smaller RATTLE}\xspace}

\title{Multi-symplectic discretisation of wave map equations}
\usepackage{authblk}

\newcommand*\email[1]{\href{mailto:#1}{\nolinkurl{#1}}}
\author[1,2]{David Cohen\thanks{\email{david.cohen@umu.se}}}
\author[1,3]{Olivier Verdier\thanks{\email{olivier.verdier@hib.no}}}
\affil[1]{Department of Mathematics and Mathematical Statistics, Umeå,  Sweden}
\affil[2]{Department of Mathematics, University of Innsbruck, Austria}
\affil[3]{Department of Computing, Mathematics and Physics, Bergen University College, Norway}
              
\begin{document}
\maketitle
\begin{abstract}
We present a new multi-symplectic formulation of constrained Hamiltonian partial differential equations, and we study the associated local conservation laws. 
A multi-symplectic discretisation based on this new formulation is exemplified by means of the Euler box scheme. 
When applied to the wave map equation, this numerical scheme is explicit, preserves the constraint and can 
be seen as a generalisation of the \Shake algorithm for constrained mechanical systems.  
Furthermore, numerical experiments show excellent conservation properties of the numerical solutions.
\end{abstract}

\begin{description}
   \item[Keywords]
	Constrained Hamiltonian partial differential equations
		· 
Wave map equations
		· 
 Multi-symplectic partial differential equation 
		· 
 Numerical discretisation
		· 
 Multi-symplectic schemes 
		·
                 Euler box scheme
\item[Mathematics Subject Classification (2010)] 
35Q51 · 35Q53 · 37K05 · 37K10 · 37M15 · 65M06 · 65M99 · 65P10
\end{description}

%%%%%%%%%%%%%%%%%%%%%%%%%%%%%%%%%%%%%%%%%%%%%%%%%%%%%%%%%%%%%%%%%%%%%

\section{Introduction}\label{sect-intro}
Ever since the seminal papers~\cite{mps98} and~\cite{br01} on multi-symplectic Hamiltonian PDEs and their discretisation, there has been a growing interest in multi-symplectic integrators.
The purpose of this paper is to propose a novel multi-symplectic integrator for multi-symplectic partial differential equations (PDEs) with constraints.
For an overview of multi-symplectic PDEs, we refer to~\cite{brid97,br01,lr04}, and references therein.

We illustrate our findings with a particular multi-symplectic PDE with constraints, the wave map equation on the sphere (or on the circle):
\begin{gather}
\label{eq:wavemap}
\begin{aligned}
u_{tt}(x,t)-\Delta u(x,t)&=\lambda u(x,t)\quad\text{in}\quad\Omega\times(0,\infty)\\
\norm{u(x,t)}^2&=1\\
u(x,0)&=u_0(x),\quad u_t(x,0)=v_0(x),
\end{aligned}
\end{gather}
where the \emph{domain} $\Omega\subset\R^m$ is an $m$-dimensional box or torus, $u\in\mathbb{S}^{\ell}$ so the \emph{target manifold} $\mathbb{S}^{\ell}$ is the $\ell$-dimensional sphere, and $\lambda$ is a Lagrange multiplier. 
Here, the initial position $u_0$ and velocity $v_0$ (living in the tangent space of the target manifold) 
of the wave are given. 
\olol{Note that we will allow the ``sphere'' to be a hyperbolic sphere, as we allow the norm on $\R^{\ell+1}$ to be degenerate.}

In the particular situation of the standard wave map equation \eqref{eq:wavemap}, 
our general multi-symplectic method takes the simple form 
\begin{gather}
   \label{eq:wavemapscheme}
   \begin{aligned}
   \frac{u^{n,i+1}-2u^{n,i}+u^{n,i-1}}{\Delta t^2}-\frac{u^{n+1,i}-2u^{n,i}+u^{n-1,i}}{\Delta x^2}&=-\olol{\lambda^{n,i+1}} u^{n,i}\\
   \norm{u^{n,i+1}}^2&=1,
   \end{aligned}
\end{gather}
where $u^{n,i}\approx u(x_n,t_i)$ on a uniform rectangular grid 
with meshes $\Delta x$ and $\Delta t$. 
This method can thus be regarded as a particular case of the \Shake algorithm for constrained mechanical systems  
\cite[Sect. VII.1.4]{HLW02}~\cite{MLMoVeWi14}, to Hamiltonian PDEs with constraints. 

The advantages of the proposed multi-symplectic method, which for the standard wave map equation reduces to \eqref{eq:wavemapscheme}, are summarized as follows:
\begin{itemize}
   % \item It is a generalization of the Shake method for 
   \item The implementation of the numerical method is effortless;
   \item It has no energy drift;
   \item It is explicit for the wave map equation on the sphere;
   \item We can simulate wave map equations with an additional smooth potential;
   \item \olol{We can handle arbitrary target Riemannian submanifolds of $\R^n$, where $\R^n$ is equipped with a nondegenerate bilinear form.}
\end{itemize}

The wave map equation \eqref{eq:wavemap}  has received considerable attention, from a more theoretical point of view,  during the last decades. 
It has applications in general relativity and in particle physics, see~\cite{ss98,r04,t04,b09} and references therein for details.
Furthermore, this equation is integrable for a domain of dimension one ($m=1$)~\cite{bfp07}, or when the target manifold is a circle;
has a conserved energy;
is time reversible; 
is non dissipative;
is invariant with respect to the scaling $u(x,t)\to u(\gamma x,\gamma t)$ for $\gamma\in\R$;
is related to Einstein equations~\cite{bcm95}; 
is related to the sine-Gordon equation~\cite{p76,ss98}; 
has critical regularity $m/2$; 
possesses, for example, global smooth solutions if the initial data are smooth and if the domain is $\R^1$~\cite{t00,t01,ss02,sz03,kt98,t05};
has blow-up solutions in finite time if the domain is $\R^m$ with $m\geq3$~\cite{k08,kst08};
it is however an open problem to show if smooth solutions do become singular in finite time;
etc. 
See also~\cite{s97,t04} and references therein.

% The literature on the numerical analysis of the wave map equation is, on the other hand, relatively scarce.
We now review the literature on the numerical analysis of the wave map equation.
The earlier papers~\cite{bct01,il02} 
report numerical evidence of finite-time blow-up of smooth initial data. 
These works restrict to equivariant maps, where the wave map equation reduces 
to a semilinear scalar wave equation. This scalar problem is then discretised 
with the standard leapfrog scheme or the Crank--Nicholson scheme with adaptive mesh.  
Very recently, there has been a renewed interest for the numerical discretisation 
of wave map problems starting with the series of papers~\cite{bfp07,b09,blp09,bps10}. 
These works prove convergence of certain (semi)-implicit finite element based methods 
to weak solutions of wave map equations. The main aim of~\cite{pf12} 
is to compare the evolution of (the blow-up of) equivariant maps using 
the classical Runge--Kutta $4$ scheme and the \Rattle algorithm for the 
time discretisation of wave map equations. 
Here, the authors used the method of lines to discretise the PDE 
and a five-point formulae for the spatial derivatives. Our multi-symplectic numerical 
method shares similarities with the \Rattle algorithm (as we shall see in \autoref{sect-app} 
that it is the \Shake algorithm in time, and \Rattle is almost identical to \Shake~\cite[\S\,5.1.2]{MLMoVeWi14}) but we would 
like to point out that our formalism is more general than the one proposed in~\cite{pf12}.    
In addition, the authors of the previously cited paper analyse, for the first time, 
the blow-up dynamics and singularity formation in the nonequivariant 
case using the same numerical methods in reference~\cite{fp13}. Furthermore, 
the recent publication~\cite{kw13} presents a 
finite difference method applied to a reformulation of the wave map equation. 
The proposed method conserves the energy, the constraint and 
converges to the weak solution of the wave map equation. 
This numerical method is however implicit. 
% We finally mention that a spatial discretization of the wave map equation is developed 
% in~\cite{ms98} in order to prove existence of solutions, but without numerical application.

The paper is organised as follows. 
\autoref{sect-msPDE} presents new multi-symplectic formulation 
and discretisation of general Hamiltonian PDEs with constraint. 
This is then illustrate for the particular case of wave map equations in \autoref{sect-app} and \autoref{sect-numexp}. 
The paper ends with concluding remarks in \autoref{sect-conc} and 
\olol{with explanations on how to simulate wave map equations where the target manifold is the 
complex projective space in \autoref{apCPS}.} 

%%%%%%%%%%%%%%%%%%%%%%%%%%%%%%%%%%%%%%%%%%%%%%%%%%%%%%%%%%%%%%%%%%%%%

\section{Multi-symplectic Hamiltonian PDEs with constraint}\label{sect-msPDE}
% In this section, we consider the numerical discretisation of partial differential equations subject to a constraint 
We begin by extending the concept of multi-symplectic PDEs to multi-symplectic PDEs with constraint. 
We will then use this new multi-symplectic formulation to derive a multi-symplectic numerical scheme for the above type of problems.

\subsection{Multi-symplectic formulation of the equations}\label{sect-mss}
There are two standard ways to construct multi-symplectic formulations of 
a PDE. One approach is using the Lagrangian formulation of the problem, 
see the early papers~\cite{g91,mps98} and references therein. The other approach is to write 
the partial differential equation as a system of equations containing only first-order derivatives 
in space and time, see equation \eqref{prob} below, and then to extract the multi-symplectic structure, 
see the early papers~\cite{brid97,br01,lr04} and references therein. 

We will now generalise this second approach to PDEs with constraints. 
In order to do this, let $n$ be an integer, two skew-symmetric 
matrices $M,K\in\R^{n\times n}$ and a scalar 
function $S\colon\R^n\to\R$. We consider \emph{Hamiltonian systems on a multi-symplectic 
structure with constraint}
\begin{gather}
   \label{prob}
   \begin{aligned}
   Mz_t+Kz_x&=\nabla_zS(z)-\lambda \GG(z)\\
   g(z)&=0.
   \end{aligned}
\end{gather}
Here, $z=z(x,t)\in\R^n$ is the state variable with components $z=(z_1,\ldots,z_n)$.
$\lambda$ is a Lagrange multiplier, $x\in[0,1]$ (for simplicity, see the remark below) and $t>0$. 
The motion is thus constrained to satisfy $g(z)=0$, where $g\colon\R^{\ell}\to\R$ and $\GG(z)$ denotes the gradient of $g$. 
Note that it is straightforward to generalize to the case of more than one constraint.

\olol{Remark that, one could add the Lagrange multiplier $\lambda$ as a variable in $z$ and add a zero 
row and column to $M$ and $K$. this would give the standard multi-symplectic formulation. 
However, in general, a scheme derived with this direct reformulation of the equation 
will not be stable.
This problem is well known in differential-algebraic equations: in general one has to enforce the constraint at the end step.}  

Observe that most multi-symplectic PDEs have removable constraints defining the auxiliary variables.
However, in our paper, the constraint is imposed externally.
% For the wave map equation, it corresponds to having a differential 
% algebraic equation of index $3$

\begin{remark}
One can further treat the case $x=(x_1,x_2)\in[0,1]^2$ (or any higher dimension) considering the multi-symplectic formulation 
\begin{align*}
Mz_t+K_1z_{x_1}+K_2z_{x_2}&=\nabla_zS(z)-\lambda \GG(z)\\
g(z)&=0
\end{align*}
with three skew-symmetric matrices $M,K_1$ and $K_2$.
See for example \autoref{sec:convergence}.
\end{remark}

\subsection{Conservation laws}\label{sect-cl}

From the formulation \eqref{prob}, we shall now introduce the conservation 
laws  of multi-symplecticity, energy and momentum. 
These derivations are similar to~\cite[Chap. 12]{lr04} with the added difficulty of 
the fulfillment of the constraint. 
\begin{proposition}\label{prop-clms}
The differential forms 
$$
\omega:=\frac12{\dd}z\wedge M{\dd}z\qquad\text{and}
\qquad\kappa:=\frac12{\dd}z\wedge K{\dd}z
$$
satisfy the following \emph{conservation law of multi-symplecticity}
\begin{align}\label{CLms}
\omega_t+\kappa_x=0 
\end{align}
along the solutions of the multi-symplectic PDE \eqref{prob}.
\end{proposition}
\begin{proof}
Let us consider the variational equation of \eqref{prob}
\begin{align*}
M{\dd}z_t+K{\dd}z_x&=S_{zz}(z){\dd}z-\dd(\lambda \GG(z))\\
\GG(z)\dd z &=0.
\end{align*}
Taking the wedge product of the above expression with ${\dd}z$, one then obtains
\begin{align*}
\dd z\wedge M{\dd}z_t+\dd z\wedge K{\dd}z_x&=\dd z\wedge S_{zz}(z){\dd}z
-\dd z\wedge(\dd\lambda \GG(z))-\dd z\wedge(g_{zz}(z)\dd z)\lambda.
\end{align*}
Using the symmetry of $S_{zz}(z)$, the symmetry of the Hessian matrix $g_{zz}(z)$, 
and using the constraint, we see that the right-hand side is equal to zero. 
Finally, applying properties of the wedge product, we observe that 
$$
\omega_t=\frac12\dd z_t\wedge M\dd z+\frac12\dd z\wedge M\dd z_t
=-\frac12M\dd z_t\wedge \dd z+\frac12\dd z\wedge M\dd z_t
=\dd z\wedge M\dd z_t, 
$$
and similarly for the term $\kappa_x$. 
This gives the above conservation law of multi-symplecticity.
\end{proof}
Observe that the conservation of symplecticity in \autoref{prop-clms} amounts to a conservation of presymplecticity on the constraint submanifold.
\olol{Note, however, that there is no consensus as to what presymplecticity and symplecticity are in the multi-symplectic case.
      In fact, most definitions of multi-symplecticity would correspond to presymplecticity in one independent variable.
      Studying in which way our methods are in fact multi-symplectic (in a stronger sense than presymplecticity) is outside the scope of this paper.}

As noted in~\cite{mr03}, the conservation law of multi-symplecticity \eqref{CLms} 
can be simplified by taking a non-unique splitting of the matrices 
$M$ and $K$ (see also \autoref{sect-msDISC} below) such that 
$$
M=M_++M_-,\qquad K=K_++K_-,
$$
where 
$$
M_+^T=-M_-\qquad\text{and}\qquad K_+^T=-K_-.
$$
Hence \eqref{CLms} holds with 
$$
\omega={\dd}z\wedge M_+{\dd}z\qquad\text{and}\qquad\kappa={\dd}z\wedge K_+{\dd}z.
$$
One next obtains the conservation law of energy by taking the usual scalar product (denoted by $\langle \cdot,\cdot\rangle$) 
of \eqref{prob} with $z_t$. 
Noting that $\langle z_t,Mz_t\rangle=0$, one gets
\begin{align*}
\langle z_t,K_+z_x+K_-z_x\rangle=\langle z_t,\nabla_zS(z)\rangle-\langle z_t,\lambda \GG(z)\rangle.
\end{align*}
Since $\langle z_t,K_+z_x+K_-z_x\rangle=\partial_x\bigl(\langle z_t,K_+z\rangle\bigr)
-\partial_t\bigl(\langle z_x,K_+z\rangle\bigr)$ and $\langle z_t,\nabla_zS(z)\rangle=\partial_t\bigl(\nabla_zS(z)\bigr)$, 
one obtains the \emph{conservation law of energy}
\begin{align}\label{CLenerg}
E_t(z)+F_x(z)=0
\end{align}
with the density functions 
\begin{align*}
E(z)&=S(z)+\langle z_x,K_+z\rangle\\
F(z)&=-\langle z_t,K_+z\rangle.
\end{align*}
Similarly, the \emph{conservation law of momentum} reads
\begin{align}\label{CLmom}
I_t(z)+J_x(z)=0
\end{align}
with the density functions 
\begin{align*}
I(z)&=-\langle z_x,M_+z\rangle\\
J(z)&=S(z)+\langle z_t,M_+z\rangle.
\end{align*}

\subsection{Multi-symplectic discretisation of Hamiltonian PDEs with constraint}\label{sect-msDISC}
The goal of this subsection is now to construct a numerical method for \eqref{prob} 
which preserves a discrete analog of the conservation law 
of multi-symplecticity \eqref{CLms}. 
% Inspired by the original definition of multi-symplectic integrators from~\cite{br01}, 
% we introduce the definition.
% \begin{definition}\label{def-msint}
% A numerical method for \eqref{prob}, written schematically as  
% \begin{align*}
% M\widehat\partial_tz^{i,j}+K\widehat\partial_xz^{i,j}&=(\nabla_zS(z^{i,j}))_{i,j}-
% \begin{pmatrix}\lambda \GG(z^{i,j})\\0\\ \vdots\\0\end{pmatrix}\\
% g(z^{i,j})&=0,
% \end{align*}
% where $\widehat\partial_x$ and $\widehat\partial_t$ are discrete versions of the partial 
% derivatives $\partial_x$ and $\partial_t$, is called a {\it multi-symplectic 
% integrator for the constrained PDE} \eqref{prob} if it possesses a discrete 
% version 
% \begin{align*}
% \widehat\partial_t\omega_{i,j}+\widehat\partial_x\kappa_{i,j}=0
% \end{align*}
% of the conservation law of symplecticity \eqref{CLms}. 
% \end{definition}

In order to do this, we first extend the Euler box scheme, see for example~\cite{mr03}, 
to constrained Hamiltonian PDE \eqref{prob}. We set 
$\Delta x=x_{n+1}-x_n, n\in\mathbb{N}$, and
$\Delta t=t_{i+1}-t_i$, for a nonnegative integer $i$.
Moreover, we define the forward and backward differences in time 
\begin{equation*}
\delta_t^+z^{n,i}=\frac{z^{n,i+1}-z^{n,i}}{\Delta t}
\qquad\hbox{and}\qquad 
\delta_t^-z^{n,i}=\frac{z^{n,i}-z^{n,i-1}}{\Delta t},
\end{equation*}
and similarly for differences in space. 
% Also, we shall need the centered differences 
% $\d_t=\hl(\delta_t^++\delta_t^-)$, 
% and $\d_x=\hl(\delta_x^++\delta_x^-)$. 

Further, we introduce a splitting of the two matrices $M$ and $K$ 
in \eqref{prob}, setting $M=M_++M_-$, 
$K=K_++K_-$ where $M_+^T=-M_-$ 
and $K_+^T=-K_-$. 
In this article, we only consider this particular splitting, keeping in mind that it is not the only possible splitting. 
We now apply the symplectic Euler method to the temporal and spatial discretisation of \eqref{prob}. 
This yields the \emph{Euler box scheme for constrained Hamiltonian PDE} \eqref{prob}
\begin{gather}
   \label{ms-euler}
   \begin{align}
   M_+\delta_t^+z^{n,i}+M_-\delta_t^-z^{n,i}+K_+\delta_x^+z^{n,i}+
   K_-\delta_x^-z^{n,i}&=\nabla_zS(z^{n,i})-\olol{\lambda^{n,i+1}}\GG(z^{n,i})\\
   g(z^{n,i+1})&=0,
   \end{align}
\end{gather}
where $z^{n,i}\approx z(x_n,t_i)$ on a uniform rectangular grid. 

To conclude this subsection, we show that the Euler box scheme \eqref{ms-euler} is a multi-symplectic integrator.
\begin{proposition}
We consider the Euler box scheme \eqref{ms-euler} with $M_+^T=-M_-$ and $K_+^T=-K_-$. 
The Euler box scheme \eqref{ms-euler} for constrained Hamiltonian PDE \eqref{prob} 
satisfies the following discrete multi-symplectic conservation law
\begin{align}\label{discretCL}
\delta_t^+({\dd}z^{n,i-1}\wedge M_+{\dd}z^{n,i})+\delta_x^+({\dd}z^{n-1,i}\wedge K_+{\dd}z^{n,i})=0.
\end{align}
In analogy to the original definition of multi-symplectic integrators from~\cite{br01}, we thus call 
this numerical method a \emph{multi-symplectic integrator for \eqref{prob}}.
\end{proposition}
\begin{proof}
The proof follows the lines of the proof of \autoref{prop-clms}, 
see also the one of~\cite[Prop.~1]{mr03} in the absence of constraints. 
We start the proof by considering the discrete variational equation 
\begin{align*}
M_+\delta_t^+{\dd}z^{n,i}+M_-\delta_t^-{\dd}z^{n,i}+K_+\delta_x^+{\dd}z^{n,i}+
K_-\delta_x^-{\dd}z^{n,i}&=S_{zz}(z^{n,i}){\dd}z^{n,i}-{\dd}(\olol{\lambda^{n,i+1}} \GG(z^{n,i}))\\
\GG(z^{n,i}){\dd}z^{n,i}&=0.
\end{align*}
Taking the wedge product of the above expression with ${\dd}z^{n,i}$, we obtain
\begin{align*}
&{\dd}z^{n,i}\wedge\Bigl(M_+\delta_t^+{\dd}z^{n,i}+M_-\delta_t^-{\dd}z^{n,i}\Bigr)
+{\dd}z^{n,i}\wedge\Bigl(K_+\delta_x^+{\dd}z^{n,i}+
K_-\delta_x^-{\dd}z^{n,i}\Bigr)=\\
&{\dd}z^{n,i}\wedge S_{zz}(z^{n,i}){\dd}z^{n,i}-
{\dd}z^{n,i}\wedge{\dd}\olol{\lambda^{n,i+1}} \GG(z^{n,i})
-
{\dd}z^{n,i}\wedge g_{zz}(z^{n,i}){\dd}z^{n,i}\olol{\lambda^{n,i+1}}.
\end{align*}
Using properties of the wedge product, the symmetry of $S_{zz}(z)$ and of the Hessian matrix 
$g_{zz}(z)$, 
and the fact that the numerical solution given by \eqref{ms-euler} satisfies the 
constraint, we end up with the discrete conservation law \eqref{discretCL}.
\end{proof}
%%%%%%%%%%%%%%%%%%%%%%%%%%%%%%%%%%%%%%%%%%%%%%%%%%%%%%%%%%%%%%%%%%%%%

\section{Applications to wave map equations}\label{sect-app}
In this section, we show that the wave map equation possesses a multi-symplectic formulation. 
Furthermore, we derive an Euler box scheme for the wave map equation and show that this multi-symplectic numerical method is explicit, 
and has a particular simple form which is closely related to the \Shake algorithm.
% Finally, we perform numerical experiments in order to illustrate the main properties of the numerical solutions.

\subsection{A multi-symplectic formulation of wave map equations}\label{sect-msWME}
The following wave map equations with a smooth potential $V$~\cite{sz03,gi03,zfl05} 
\begin{gather}
\label{waveMnl}
   \begin{aligned}
   u_{tt}-u_{xx}&=-V'(u)+\lambda \GG(u)\\
   g(u)&=0,
   \end{aligned}
\end{gather}
where $u=(u_1,u_2,u_3)\in\R^3$, 
can be put into the multi-symplectic framework \eqref{prob}. 
For simplicity, we will only consider a domain in $\mathbb{R}^1$ here.  
An example on a $2$-dimensional torus will be given in \autoref{sect-numexp}.

Indeed, considering the vector of state variable $z=(u_1,u_2,u_3,v_1,v_2,v_3,m_1,m_2,m_3)$, 
taking the skew-symmetric matrices ($I$ denotes the identity matrix in $\R^3$)
$$
M=\begin{pmatrix}0 & -I & 0 \\ I & 0 & 0 \\ 0& 0& 0 \end{pmatrix}
\qquad\text{and}\qquad K=\begin{pmatrix} 0& 0& -I  \\ 0& 0& 0 \\ I& 0& 0\end{pmatrix}
$$ 
and considering the scalar function $S(z)=\frac12(v_1,v_2,v_3)^T(v_1,v_2,v_3)-
\frac12(m_1,m_2,m_3)^T(m_1,m_2,m_3)+V(u_1,u_2,u_3)$ we obtain the equivalent 
representation \eqref{prob}.
This multi-symplectic formulation of the wave map equation \eqref{waveMnl} takes the explicit form  
\begin{align*}
-v_t-m_x&=V'(u)-\lambda \GG(u)\\
u_t&=v\\
u_x&=-m\\
g(u)&=0.
\end{align*}
In particular, taking $V\equiv0$ and $g(u)=\abs{u}^2-1$ in \eqref{waveMnl}, 
one gets a multi-symplectic formulation \eqref{prob} of the classical 
wave map problem into the unit sphere~\cite{s97}
\begin{gather}
   \label{waveM}
   \begin{aligned}
   u_{tt}-u_{xx}&=\lambda u\\
   \abs{u}^2&=1.
   \end{aligned}
\end{gather}

For the wave map equation \eqref{waveMnl}, we choose the splitting of the matrices
$$
M_+=\begin{pmatrix}0 & -I &0 \\ 0&0&0 \\ 0&0&0\end{pmatrix}
\qquad\text{and}\qquad K_+=\begin{pmatrix}0 & 0 &-I \\ 0&0&0 \\ 0&0&0\end{pmatrix}.
$$
The conservation laws of multi-symplecticity, energy and momentum then read
\begin{gather}
\label{CLw}
   \begin{aligned}
   ({\dd}u\wedge{\dd}v)_t+({\dd}u\wedge{\dd}m)_x&=0\\
   (\frac12v^Tv-\frac12m^Tm+V(u)-(m_x)^Tm)_t+((u_t)^Tm)_x&=0\\
   ((u_x)^Tv)_t+(\frac12v^Tv-\frac12m^Tm+V(u)-(u_t)^Tv)_x&=0.
   \end{aligned}
\end{gather}

Integrating these two last conservation laws over the spatial domain and using appropriate 
boundary conditions, one obtains two conserved quantities. 
Wave map problems \eqref{waveMnl} are thus Hamiltonian 
PDEs with constraint and having the following conserved quantities, 
see also~\cite{t04,bfp07}, 
\begin{align}
H(u)&=\int_{\Omega}\bigl(\frac12 \abs{u_t}^2+
\frac12\abs{u_x}^2+V(u)\bigr)\,{\dd}x\label{hamwaveMnl}\\
M(u)&=\int_{\Omega}\frac12 (u_x)^Tu_t\,{\dd}x.\label{momwaveMnl}
\end{align}

\subsection{A multi-symplectic scheme for wave map equations}\label{sect-ebWME}
For the particular case of wave map problems \eqref{waveMnl}, one can eliminate all the additional variables 
in the Euler box scheme \eqref{ms-euler} and express the numerical scheme only in terms of $u$.
This gives us the following multi-symplectic integrator for wave map equations \eqref{waveMnl}
\begin{align*}
\delta_t^+\delta_t^-u^{n,i}-\delta_x^+\delta_x^-u^{n,i}&=-V'(u^{n,i})-\olol{\lambda^{n,i+1}} \GG(u^{n,i})\\
g(u^{n,i+1})&=0.
\end{align*}
Developing all the above terms, the Euler box scheme for wave map equations \eqref{waveMnl} thus reads
\begin{gather}
   \begin{aligned}
   \frac{u^{n,i+1}-2u^{n,i}+u^{n,i-1}}{\Delta t^2}-\frac{u^{n+1,i}-2u^{n,i}+u^{n-1,i}}{\Delta x^2}
   &=-V'(u^{n,i})-\olol{\lambda^{n,i+1}} \GG(u^{n,i})\\
   g(u^{n,i+1})&=0.
   \end{aligned}
\end{gather}
It is more convenient to rewrite it in the equivalent form (with a slight abuse of notation for the Lagrange multiplier $\olol{\lambda^{n,i+1}}$):
\begin{gather}
   \label{EB}
   \begin{aligned}
   \frac{\widetilde{u^{n,i+1}}-2u^{n,i}+u^{n,i-1}}{\Delta t^2}-\frac{u^{n+1,i}-2u^{n,i}+u^{n-1,i}}{\Delta x^2}
   &=-V'(u^{n,i})\\
   u^{n,i+1} &= \widetilde{u^{n,i+1}} - \olol{\lambda^{n,i+1}} \GG(u^{n,i})\\
   g(u^{n,i+1})&=0.
   \end{aligned}
\end{gather}
The last formulation emphasizes that the computation consists of two steps:
\begin{enumerate}
   \item Compute $\widetilde{u^{n,i+1}}$ using the explicit formula \eqref{EB}
   \item Project $\widetilde{u^{n,i+1}}$ onto the constraint manifold in the direction $\GG(u^{n,i})$.
\end{enumerate}
In the classical wave map case, the constraint manifold is a sphere of radius one. 
The value of the Lagrange multiplier $\olol{\lambda^{n,i+1}}$ is thus a solution of a quadratic problem. 

We assume that $u^0$ (the first step of the scheme) lies on the sphere of radius one. 
One then computes a point $\widetilde{u^1}$ by ignoring the constraint, see \autoref{fig:proj}. 
If we first define $p=\norm[\big]{\widetilde{u^1}}^2-1$ and $s = \langle u^0,\widetilde{u^1}\rangle$, 
we straightforwardly obtain
\begin{align}\label{eq:lambdaexplicit}\lambda=-s+\sqrt{s^2-p}.\end{align} 
Note that, as $s$ is generally positive, we use the following equivalent formula 
in order to avoid potential ``catastrophic cancellation'' issues: 
$\lambda = \frac{p}{-s - \sqrt{s^2 - p}}$.
With either of those formulas the projection step is the explicit operation
\begin{equation}
   u^1 = \widetilde{u^1} - \frac{p}{-s - \sqrt{s^2-p}}u^0
\end{equation}

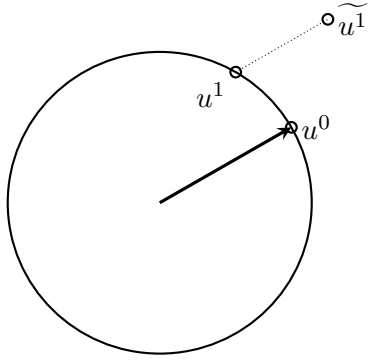
\begin{figure}
   \centering
   \begin{tikzpicture}
      \draw[thick] (0,0) circle[radius=2cm]; 
      \coordinate (q0) at (30:2cm);
      \draw[very thick, ->] (0,0) -- (q0);
      \coordinate (q1) at (60:2cm);
      \coordinate (q1t) at ($(q1) + .7*(q0)$);
      \draw[thick] (q0) circle[radius=2pt];
      \draw[thick] (q1) circle[radius=2pt];
      \draw[thick] (q1t) circle[radius=2pt];
      \draw[densely dotted] (q1) -- (q1t);
      \node[right] at (q1t) {$\widetilde{u^1}$};
      \node[below left] at (q1) {$u^1$};
      \node[right] at (q0) {$u^0$};
   \end{tikzpicture}
   \caption[]{
      The point $u^1$ is obtained by first computing a point $\widetilde{u^1}$ by ignoring the constraint.
      We then project the point $\widetilde{u^1}$ \emph{along the direction of $u^0$}, to obtain a point $u^1$ which fulfills the constraint.
      This means that we have $u^1 = \widetilde{u^1} + {\lambda} u^0$ for some scalar $\lambda$.
      In the case of a quadratic constraint, the expression for ${\lambda}$ is explicitly given by \eqref{eq:lambdaexplicit}.
       }
\label{fig:proj}
\end{figure}

\olol{The numerical integrator \eqref{EB} can also be seen as a particular instance of the 
\Shake algorithm for constrained mechanical systems, see e.g. 
\cite{MLMoVeWi14} or~\cite[Sect. VII.1.4]{HLW02}, to wave map equations. 
In other words, \eqref{EB} corresponds to an application of \Shake to a finite difference discretisation 
of the wave equation by central finite differences. }

\dada{
One can wonder what happens to the hidden constraints.
As the algorithm \eqref{EB} for wave map equations is written in $u$ only, the hidden constraints do not really make sense anymore.
Notice, however, that the constraints in $u$ are exactly preserved.
Suppose that one had used the full Euler box scheme \eqref{ms-euler} instead, with unknown variable $z$.
Then the time and space momenta, would be first order finite difference approximations of the time and space derivatives of the position $u$.
The corresponding hidden constraints would thus be approximately preserved up to first order.

Observe that for other constrained PDEs (Hamiltonian or not), it is extremely difficult to give a precise definition of hidden constraints \cite{Se10, Ve14}.
% This is for instance studied in the monograph ``Involution'' by Seiler, or in ``Reductions of Operator Pencils'' by Verdier.
}

% We conclude this subsection by mentioning a closely related 
% problem studied in~\cite{hh03}: a numerical study of the motion of a 
% two-body problem on the sphere using the \Rattle algorithm (a related formulation 
% of the \Shake algorithm). 

%%%%%%%%%%%%%%%%%%%%%%%%%%%%%%%%%%%%%%%%%%%%%%%%%%%%%%%%%%%%%%%%%%%%%%%%%%%%%%%%%%%%
\section{Numerical experiments for wave map equations}
\label{sect-numexp}
This section illustrates the main properties of the Euler box scheme \eqref{EB} when applied to the wave 
map equations \eqref{waveMnl} and \eqref{waveM}. 

These numerical experiments illustrate the following properties of our method:
\begin{enumerate}
   \item We observe convergence of order two, and absence of energy drift for smooth solutions (\autoref{sec:convergence});
   \item We observe breather solutions accurately for several periods (\autoref{sec:breather});
   \item We observe the correct blow-up time as in~\cite{kw13,bfp07} (\autoref{sec:blowup});
   \item We show that we can simulate the wave map equation with potential (\autoref{sec:potential});
   \item \olol{We show the versatility of our approach by considering the Poincaré disk or 
   the complex projective space as a target manifold, see also \autoref{apCPS}.}
\end{enumerate}
%%%%%%%%%%%%%%%%%%%%%%%%%%%%%%%%%%%%%%%%%%%%%%%%%%%%%%%%%%%%%%%%%%%%%%%%%%%%%%%%%%%%
% \subsubsection{Wave map equation on the circle}
% Let us first consider problem \eqref{waveM}  
% \begin{align*}
% u_{tt}(x,t)-u_{xx}(x,t)&=\lambda u(x,t)\quad\text{in}\quad[0,2\pi]\times(0,\infty)\\
% \abs{u(x,t)}&=1,
% \end{align*}
% where $u=u(x,t)\in\R^2$ with periodic boundary 
% conditions $u(0,t)=u(2\pi,t)$ and initial values 
% \begin{align*}
% \end{align*}

%%%%%%%%%%%%%%%%%%%%%%%%%%%%%%%%%%%%%%%%%%%%%%%%%%%%%%%%%%%%%%%%%%%%%%%%%%%%%%%%%%%%

\subsection{Convergence rates and approximate energy conservation for the wave map from the torus to the circle}
\label{sec:convergence}
We consider the wave map problem \eqref{waveM} in two spatial dimensions~\cite{kw13}
\begin{gather}
   \label{waveM2d}
   \begin{aligned}
   u_{tt}-u_{x_1x_1}-u_{x_2x_2}&=\lambda u\\
   \abs{u}^2-1&=0,
   \end{aligned}
\end{gather}
where $u=u(x_1,x_2,t)\in\R^2$, with $(x_1,x_2)\in\mathbb{T}^2$ the $2$-dimensional torus. 

For sake of completeness let us first state the multi-symplectic formulation 
and the scheme in the present setting. The above wave map problem has the following multi-symplectic formulation
\begin{align*}
Mz_t+K_1z_{x_1}+K_2z_{x_2}&=\nabla_zS(z)-
\lambda \GG(u)\\
g(u)&=0
\end{align*}
with the state variable $z=(u,v,p_1,p_2)$, the function 
$S(z)=\frac12v^Tv-\frac12p_1^Tp_1-\frac12p_2^Tp_2$, the constraint 
$g(u)=\abs{u}^2-1$ and the three skew-symmetric matrices 
$$
M=\begin{pmatrix}0 & -I & 0 &0\\ I & 0 & 0 &0\\ 0& 0& 0 &0\\ 0& 0& 0&0\end{pmatrix}
\qquad\text{and}\qquad K_1=\begin{pmatrix} 0& 0& -I &0 \\ 0& 0& 0&0 \\ I& 0& 0&0\\ 0& 0& 0&0\end{pmatrix} 
\qquad\text{and}\qquad K_2=\begin{pmatrix} 0& 0& 0 &-I \\ 0& 0& 0&0 \\ 0& 0& 0&0\\ I& 0& 0&0\end{pmatrix}.
$$
The corresponding multi-symplectic Euler box scheme, for the classical splitting of the matrices, reads
\begin{gather}
      \label{eq:multishake2d}
   \begin{aligned}
   \delta_t^+\delta_t^-u^{n,m,i}-\delta_{x_1}^+\delta_{x_1}^-u^{n,m,i}-\delta_{x_2}^+\delta_{x_2}^-u^{n,m,i}&=-\olol{\lambda^{n,i+1}} \GG(u^{n,m,i})\\
   g(u^{n,m,i+1})&=0.
   \end{aligned}
\end{gather}

Problem \eqref{waveM2d} has the following analytical solution: 
\begin{align}\label{eq:analytic}
u(x,t) = \paren[\Big]{\cos\paren[\big]{\theta(x,t)}, \sin\paren[\big]{\theta(x,t)}}
\end{align}
where $x=(x_1,x_2)$ and $\theta$ is a solution of the \emph{linear} wave equation
\begin{align}
   \theta_{tt} - \triangle \theta = 0.
\end{align}
Such solutions are superpositions of the functions
\begin{align}\label{eq:superposition}
   \theta_k(x,t) \coloneqq a_k \cos\paren[\big]{k_1 x_1 + k_2 x_2 - \norm{k} t - \varphi_k},
\end{align}
where $k = (k_1,k_2)\in\Z^2$ is the {wavenumber}, $a_k\in\R$ is the amplitude, 
and $\varphi_k\in\mathbb{T}$ is an arbitrary phase shift.

In the following numerical experiments, we thus compute the exact solution of our wave map problem \eqref{waveM2d} using 
formulas \eqref{eq:analytic} and \eqref{eq:superposition} and choosing the values $a_k$, $\varphi_k$ from 
\autoref{tab:ampphase}.

\begin{table}
   \centering
   \begin{tabular}{@{}lll@{}} \toprule
      Wavenumber &  Amplitude & Phase \\
      \midrule
      $(1,1)$ & $1$ & $0$ \\
      $(2,1)$ & $0.5$ & $0.5$ \\
      $(-1,1)$ & $0.2$ & $0.8$ \\
      \bottomrule
   \end{tabular}
   \caption[]{The values of the wavenumbers (pairs of integers), as well as amplitudes (scalar) 
and phase shifts (angle) used in \autoref{fig:order}.}
   \label{tab:ampphase}
\end{table}

We now use our multi-symplectic numerical method \eqref{eq:multishake2d}.
\autoref{fig:order} shows a plot of the error, i.e., the norm of the difference between 
the computed solution and the exact solution.
The norm used is that of the space $\mathrm{L}^{\infty}(0,T; \mathrm{L}^2(\mathbb{T}^2))$, where $\mathbb{T}^2$ is the spatial domain, the two-dimensional torus.
The integer $N$ denotes the number of points in space.
The final time is $T = 1$.
We choose a Courant ratio $\Delta t/\Delta x=1/2$, i.\,e., there are twice as many time points than space points.
The slope of the fitted grey line is $2.15$ which indicates convergence of order two.

% The Courant Number is chosen at $1/2$, i.\,e, there are twice as many time points than space points. 
% The norm used in the computation of the error is that of the space $\mathrm{L}^{\infty}(0,T; \mathrm{L}^2(\mathbb{T}^2))$, 
% where $\mathbb{T}^2$ is the spatial domain, the two-dimensional torus. The error of the numerical method at time $T=1$ 
% is shown in \autoref{fig:order}. For this smooth solution, a convergence of order $2$ is observed. 

\begin{figure}
   \centering
   \includegraphics[width=.8\textwidth]{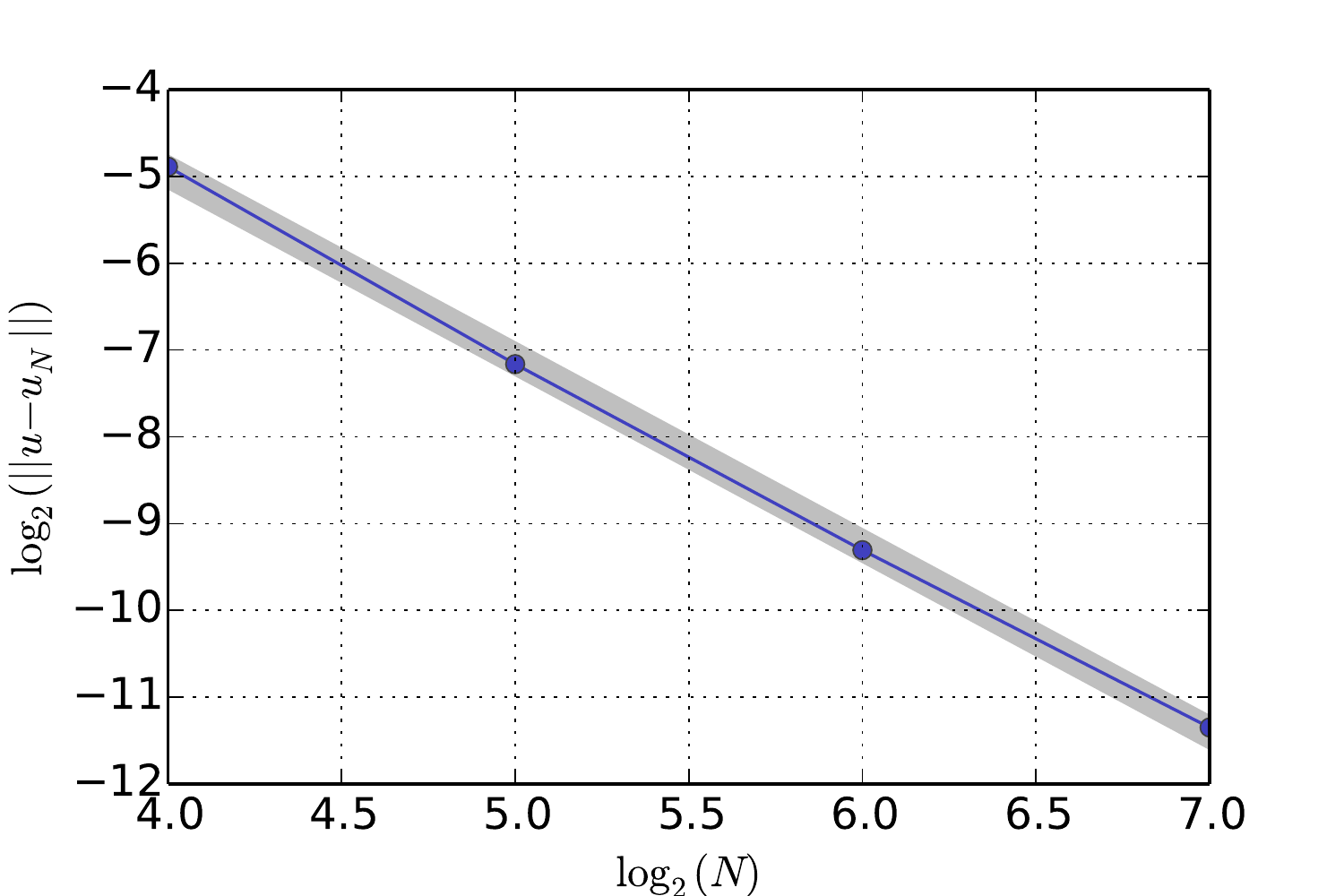}
   \caption{Wave map from the torus $\mathbb{T}^2$ to the circle: plot of the  error of the computed solution with respect to the exact one. 
    The integer $N$ denotes the number of points in space. The slope of the fitted grey line is $2.15$ which indicates convergence of order two.}
\label{fig:order}
\end{figure}

\autoref{fig:energy_wave} displays the relative energy error between the energy $E$ and the initial energy $E_0 = 66.3$, along the numerical solution given by the multi-symplectic scheme \eqref{eq:multishake2d} on the time interval $[0,11]$ with $N=2^7$ points in space. 
We observe good approximate energy conservation.

\begin{figure}
   \centering
   \includegraphics[width=.8\textwidth]{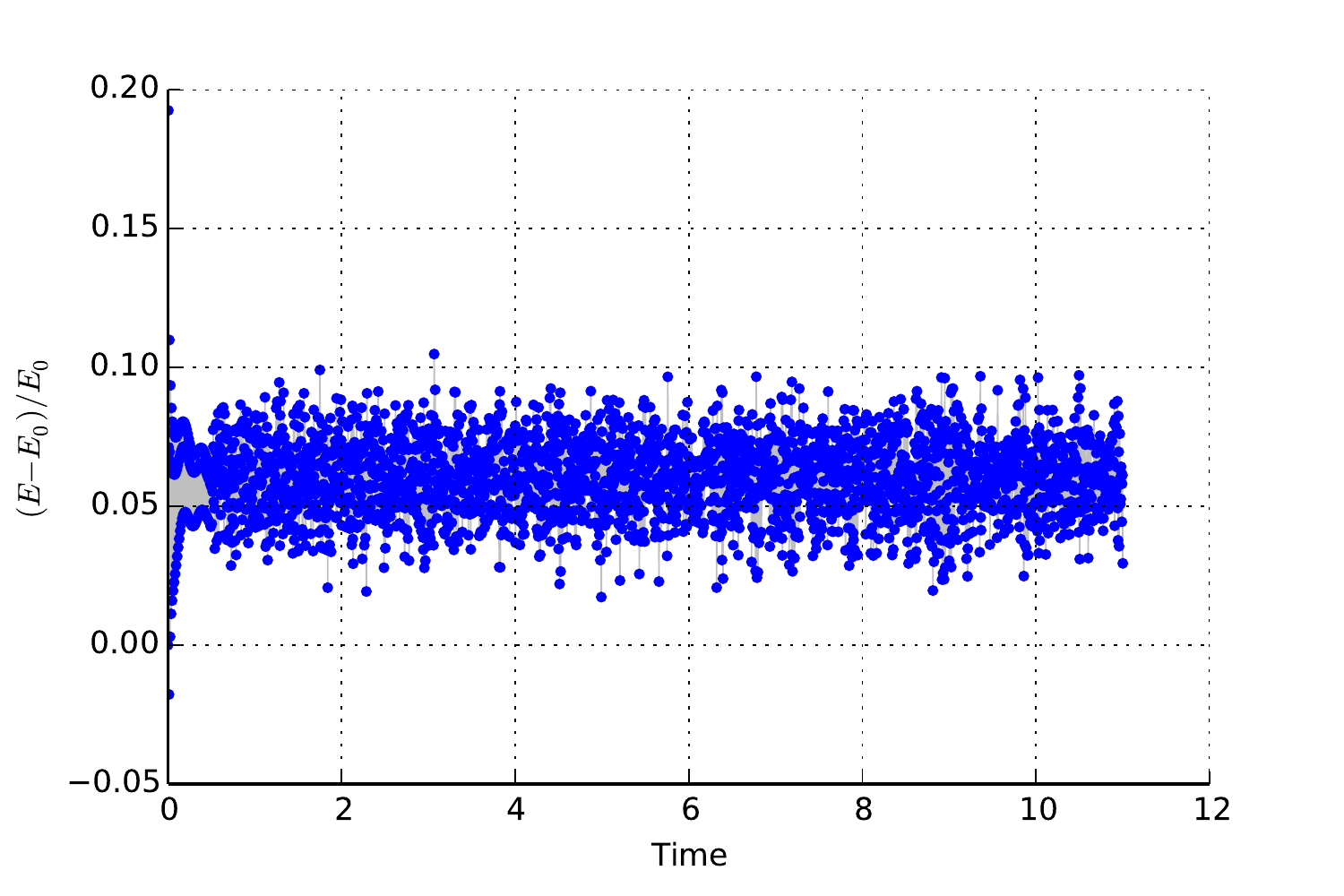}
   \caption{Wave map from the torus $\mathbb{T}^2$ to the circle: relative energy error $(E-E_0)/E_0$ along the numerical solution on the time interval $[0,11]$.
      The initial energy is $E_0 = 66.3$.
   }
   \label{fig:energy_wave}
\end{figure}

%%%%%%%%%%%%%%%%%%%%%%%%%%%%%%%%%%%%%%%%%%%%%%%%%%%%%%%%%%%%%%%%%%%%%%%%%%%%%%%%%%%%
\subsection{Breather solutions}
\label{sec:breather}

\newcommand*\ellj{\sqrt{\ell^2 - j^2}}

We now consider breather solutions of the wave map equation
\begin{align*}
u_{tt}-u_{xx}&=\lambda u\\
\abs{u}^2&=1,
\end{align*}
where $u\colon\R\times \mathbb{T} \to\R^3$. We consider the following initial condition
\begin{subequations}
\label{breatherInit}
\begin{align}
u_0&=\begin{pmatrix}\cos(\ell x)\\ \sin(\ell x)\\0 \end{pmatrix} \label{eq:breather_zero}\\
\label{eq:breather_one}
u_1&= u_0 + \begin{pmatrix}0\\0\\ \varepsilon \sin(jx) \end{pmatrix}, 
\end{align}
\end{subequations}
where $\ell$ and $j$ are integers such that $1 \leq j \leq \ell-1$, and $\varepsilon$ is an arbitrary small parameter.
Observe that the period of the breather wave map tends to infinity when $\varepsilon$ goes to zero.

Let us first show that this initial condition is a first order approximation of the breather wave map 
given by~\cite[Lemma~7.2]{ss96}. Indeed, this breather wave map is given by
$$
u_b(x,s)\coloneqq\begin{pmatrix}
\cos(\kappa)\cos(\ell x-\kappa)-\sin(\kappa)\cos(s/\sin(\kappa))\sin(\ell x-\kappa)\\
\cos(\kappa)\sin(\ell x-\kappa)+\sin(\kappa)\cos(s/\sin(\kappa))\cos(\ell x-\kappa)\\
\sin(\kappa)\sin(s/\sin(\kappa))
\end{pmatrix},
$$
where
\begin{align*}
\tan(\kappa)=\frac{\ell}{j}\tan(jx)\qquad\text{and}\qquad s(x,t)\coloneqq \int_{-\infty}^t \ell\sin\paren[\big]{\alpha(\tau,x)/2} \dd \tau
\end{align*}
and 
$$
\alpha=\alpha(x,t)=4\arctan\paren[\bigg]{\frac{\sqrt{\ell^2-j^2}}{j}\frac{\sin(jx)}{\cosh(\sqrt{\ell^2-j^2}t)}}
$$
is the classical breather solution of the generalised sine-Gordon equation $\alpha_{tt}-\alpha_{xx}-\ell^2\sin(\alpha)=0$. 

In fact, the initial condition  \eqref{breatherInit} is obtained using 
a first order approximation of $u_b(x,s)$ at $s=0$.
First, as noted in~\cite{ss96}, using the identity for the sum of angles of trigonometric functions, the value of $u_b(x,0)$ simply reduces to \eqref{eq:breather_zero}.

Now, for $t \simeq -\infty$, we have $\cosh(\ellj t) \simeq \infty$, so we approximate $\alpha(x,t)$ by
\begin{align*}
   \alpha(x,t) \simeq 4 \frac{\ellj}{j} \exp(\ellj t) \sin(jx)
   .
\end{align*}
This gives in turn
\begin{align*}
   s(x,t) \simeq \int_{-\infty}^{t} \frac{\ell}{2} \alpha(x,\tau) \dd \tau
\end{align*}
so we obtain
\begin{align*}
   s(x,t) \simeq 2 \frac{\ell}{j} \exp\paren[\big]{\ellj t} \sin(jx)
   ,
\end{align*}
and we choose
\begin{align*}
   \varepsilon \coloneqq 2 \frac{\ell}{j} \exp\paren{\ellj  t},
\end{align*}
which is infinitesimally small when $t \simeq -\infty$.

Finally, a first order development of $u_b$ at $s=0$ yields
\begin{align*}
u_b(x,s) \simeq u_b(x,0) + \begin{pmatrix}
   0 \\ 0 \\ s
\end{pmatrix}
\end{align*}
which justifies the choice \eqref{eq:breather_one}.

We now run our multi-symplectic scheme \eqref{EB} on the example corresponding to the winding number $\ell=7$ 
and the initial frequency $j=5$. The value of $\varepsilon$ in \eqref{eq:breather_one} is set to $10^{-4}$. 
\autoref{fig:breathers} presents snapshots of the numerical solutions computed with $N=2^9$ points in space, and a Courant ratio $\Delta t/\Delta x = 1/2$.
We observe a periodic motion, which leads us to define a \emph{period} as the first time at which the numerical solution returns to its initial state.
% We observe, as expected, that the breather orbits are unstable, (in this case after two periods), as shown in \autoref{fig:breatherinstability}.
      
\newcommand*\subfloatbreather[2]{%
\subfloat[#2]{\includegraphics[width=.2\textwidth]{cropped/_breather_#1}}
}

\begin{figure}
   \centering
   \subfloatbreather{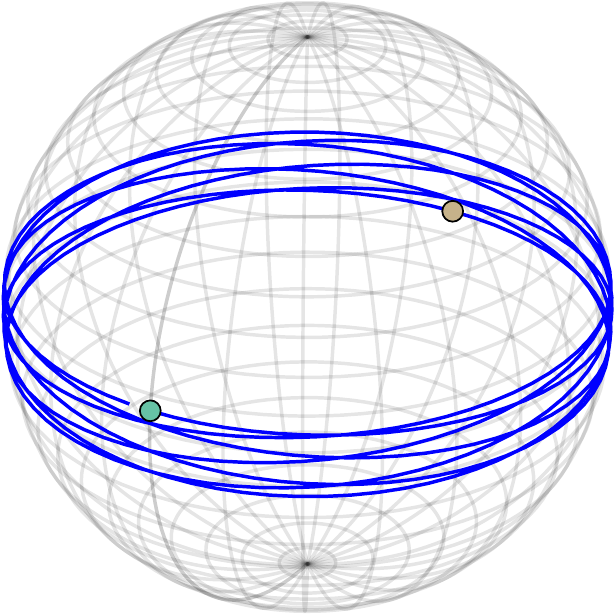}{0.262}\quad
   \subfloatbreather{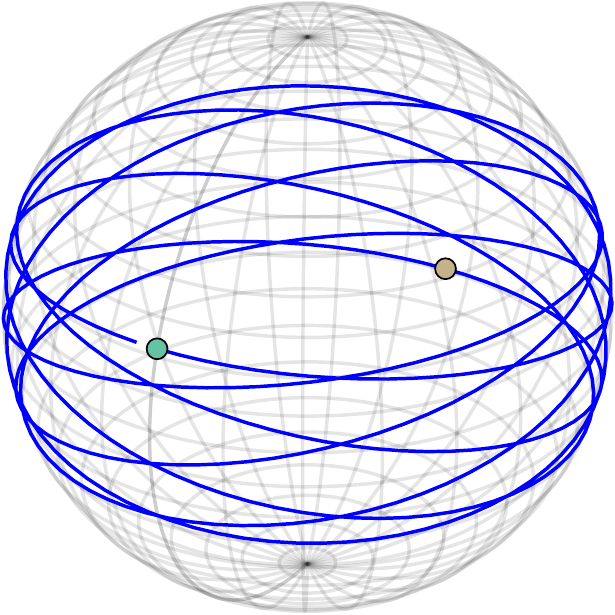}{0.325}\quad
   \subfloatbreather{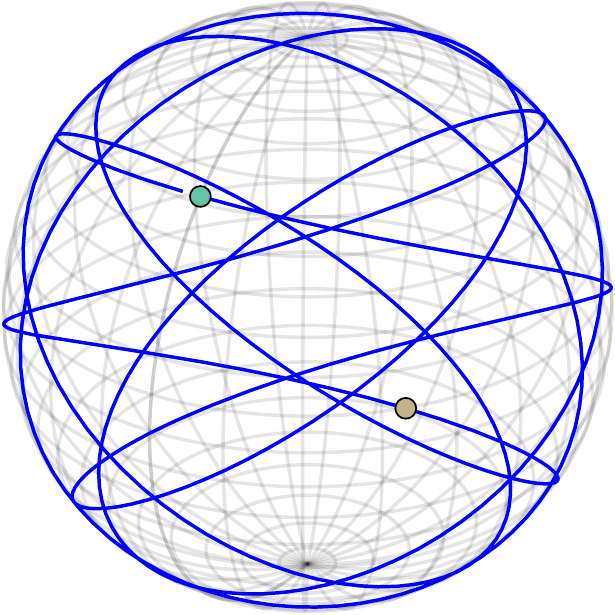}{0.384}\quad
   \subfloatbreather{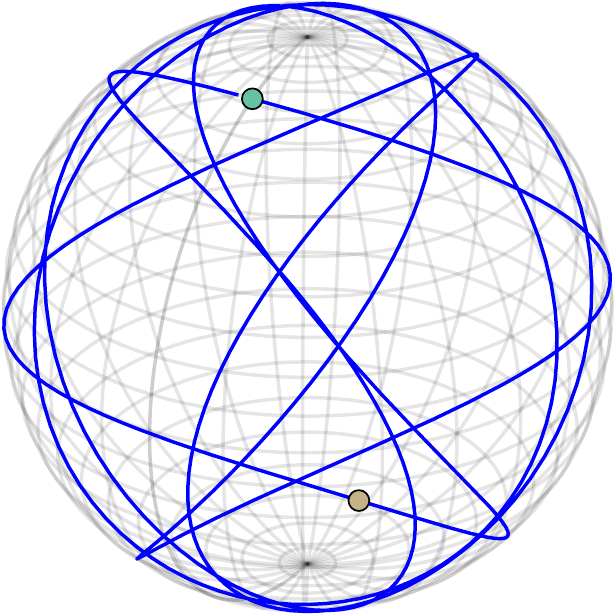}{0.410}\\
   \subfloatbreather{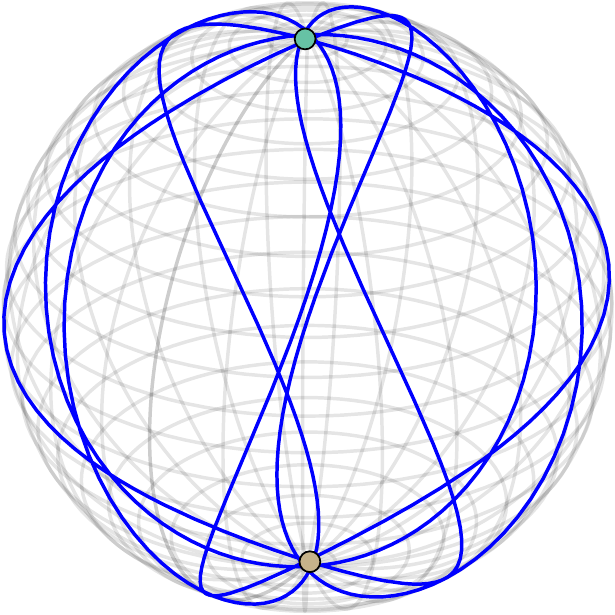}{0.428}\quad
   \subfloatbreather{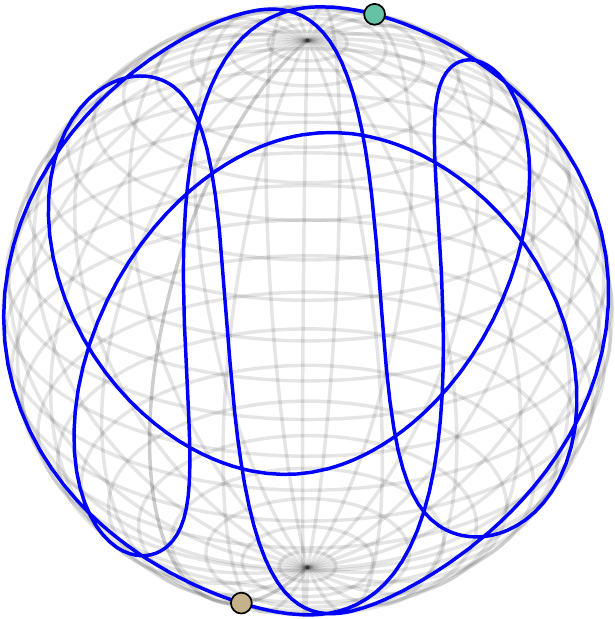}{0.450}\quad
   \subfloatbreather{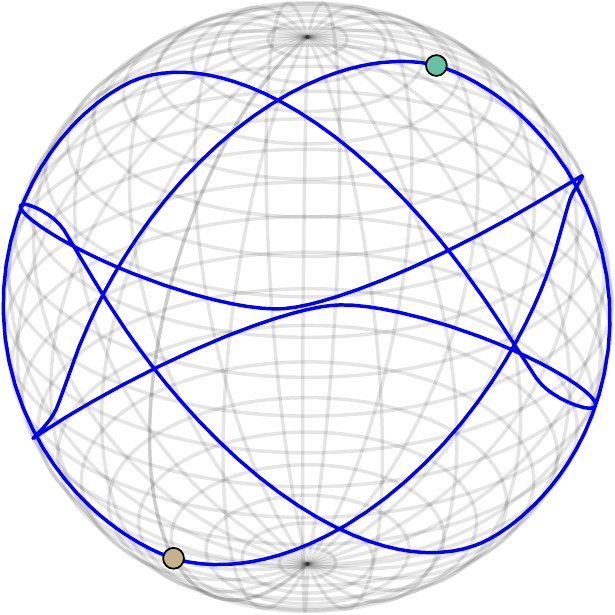}{0.476}\quad
   \subfloatbreather{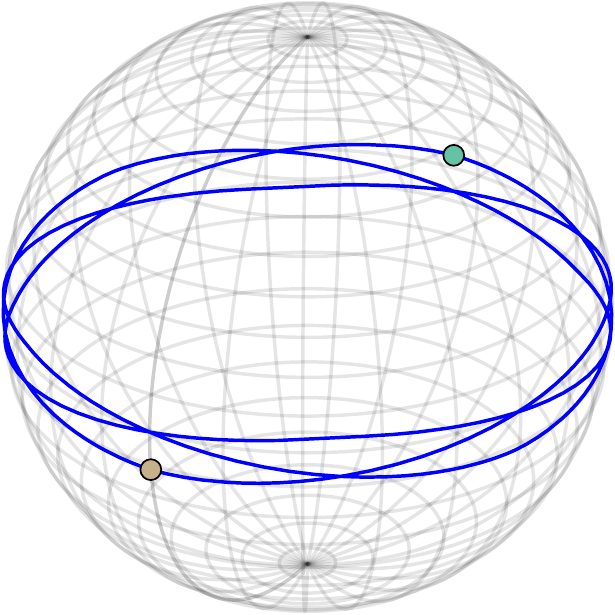}{0.494}

   \caption[]{Wave map from the circle to the sphere: snapshots of the breather of winding number $\ell=7$ and initial frequency $j=5$.
             The caption below each snapshot indicates the time in period units.
             One can further see that each particle stays on a circle on the sphere, thus illustrating~\cite[Lemma~7.1]{ss96}.
          }
\label{fig:breathers}
\end{figure}

Using the same data as in the previous numerical experiments, \autoref{fig:breatherinstability} displays the relative energy error and amplitude in the $z$ direction of the numerical solution over three periods.
These plots, in period units, show, as expected, that the breathers are not stable. 
Note that it is a major merit of the proposed numerical method to be able to accurately 
compute the breather over a few periods. 

\begin{figure}
   \centering
   \begin{tikzpicture}
      \begin{axis}
         [%
         axis on top,
         enlargelimits=false, 
         ylabel={Maximum amplitude in $z$},
         xlabel={Time},
         grid=major,
         % ytick align=outside,
         % ytick={0,0.1,.2,0.28,.4,.5},
         ]
         \addplot
         graphics
         [xmin=0.0,ymin=0.0,xmax=3,ymax=1]
         {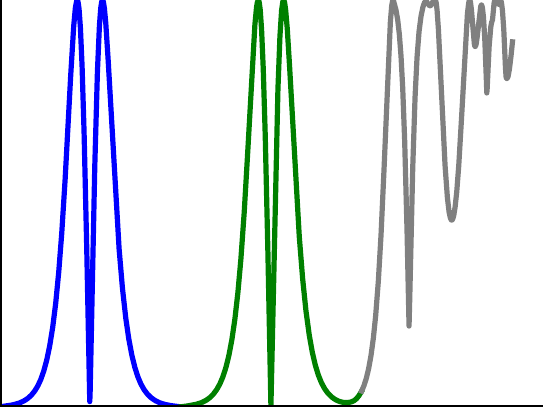};
         
      \end{axis}
   \end{tikzpicture}
   \caption{
   % Plot of the maximum amplitude in the direction orthogonal to the initial circle.
      Wave map from the circle to the sphere: plot of the maximum amplitude in the $z$ direction over three periods.
            Each period is plotted with a different colour.
            Observe how the proposed numerical method accurately computes the breathers over a few periods, despite the unstability of the solutions.
% We observe, as expected, that the breathers are not stable. 
% Note that it is a major merit of the proposed numerical method to be able to accurately compute the breather over a few periods. 
         }
\label{fig:breatherinstability}
\end{figure}

Finally, \autoref{fig:breatherenergy} shows the relative energy error of the above breather over thirty period units. 
We use the same colours as in \autoref{fig:breatherinstability}.  
The initial energy is still $E_0 = 967$, so we see that the energy oscillations are minimal, and that there is no energy drift. 

\begin{figure}
   \centering
   \begin{tikzpicture}
      \begin{axis}
         [%
         axis on top,
         grid=major,
         enlargelimits=false, 
         xlabel={Time in period units},
         ylabel={$(E-E_0)/E_0$},
         ytick align=outside,
         ]
         \addplot
         graphics
         [xmin=0,ymin=-0.015,xmax=30,ymax=0.015]
         {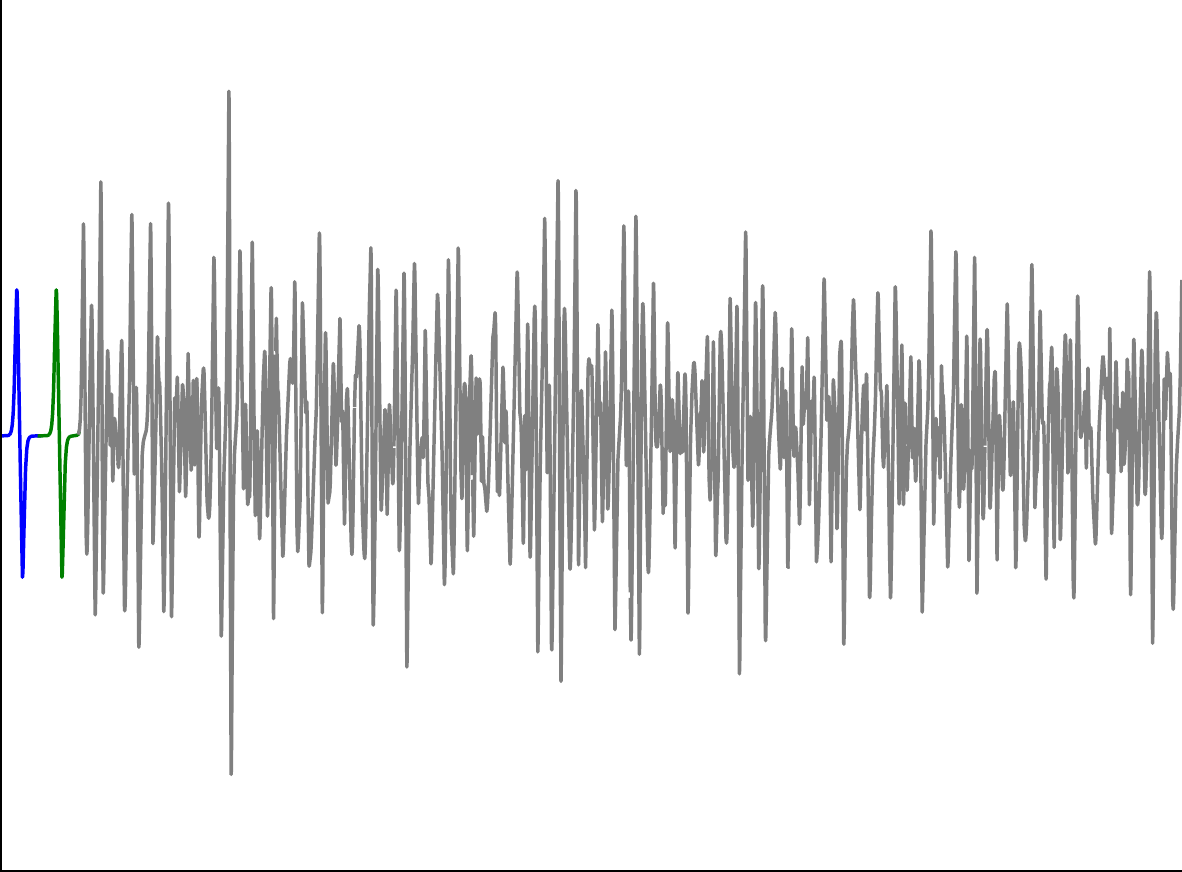};
\end{axis}
   \end{tikzpicture}
   \caption{Wave map from the circle to the sphere: plot of the relative energy error $(E-E_0)/E_0$ of the breather of \autoref{fig:breatherinstability} over thirty period units.
      We observe no energy drift.
   }
   \label{fig:breatherenergy}
\end{figure}

%%%%%%%%%%%%%%%%%%%%%%%%%%%%%%%%%%%%%%%%%%%%%%%%%%%%%%%%%%%%%%%%%%%%%%%%%%%%%%%%%%%%
\subsection{Blow-up of smooth initial data}
\label{sec:blowup}
The purpose of this section is to show that our scheme obtains 
the same blow-up time as in the numerical experiments presented in~\cite{bfp07,kw13}. 
\olol{This is only a preliminary study and one should be aware that adaptive mesh refinement 
techniques,~\cite{MR1374282} for example, should be used in order to have a proper 
understanding of the behaviour of the solution close to blow-up times.} 

We consider the wave map equation 
\begin{align*}
u_{tt}-\Delta u&=\lambda u\\
\abs{u}^2&=1,
\end{align*}
where $u=u(x,t)\in\R^3$, $x=(x_1,x_2)\in[-1/2,1/2]^2$. This problem is supplemented with 
homogeneous Neumann boundary conditions. 

We consider the equivariant initial data~\cite{bfp07,kw13} 
\begin{align}\label{busid}
u_0(x_1,x_2)=
\frac{1}{a(r)^2+r^2}\paren[\Big]{2x_1a(r),2x_2a(r),a(r)^2-r^2}
\end{align}
with $r=\sqrt{x_1^2+x_2^2}$ and
\begin{align*}
a(r)=
\begin{cases}
(1-2r)^4 & r \leq 1/2\\
0 & r \geq 1/2.
\end{cases}
\end{align*}

We use our multi-symplectic scheme with $N=128$ points in each direction in space, with a Courant ratio $\Delta t/\Delta x = 1/2$. 
Our results are shown in \autoref{fig:blowup_time} and \autoref{fig:blowup_radius}. 
Looking at the two subplots of \autoref{fig:blowup_time}, we can estimate 
the blow-up time at $0.28$ by glancing at the maximum value of the computed energy. This corresponds 
to the time where the center particle brutally flips over to pointing to the opposite direction $z=-1$. 
\autoref{fig:blowup_radius} offers a view of the $z$ coordinate versus a radius.
Once again we observe the blow-up time at $0.28$ represented by a red horizontal line.

\begin{figure}
   \centering
   \subfloat[Relative energy error versus time]{
   \begin{tikzpicture}[scale=.7]
      \begin{axis}
         [%
         axis on top,
         grid=major,
         enlargelimits=false, 
         xlabel={Time},
         ylabel={$(E-E_0)/E_0$},
         ytick align=outside,
         ]
         \addplot
         graphics
         [xmin=0.0,ymin=-1,xmax=.6,ymax=5]
         {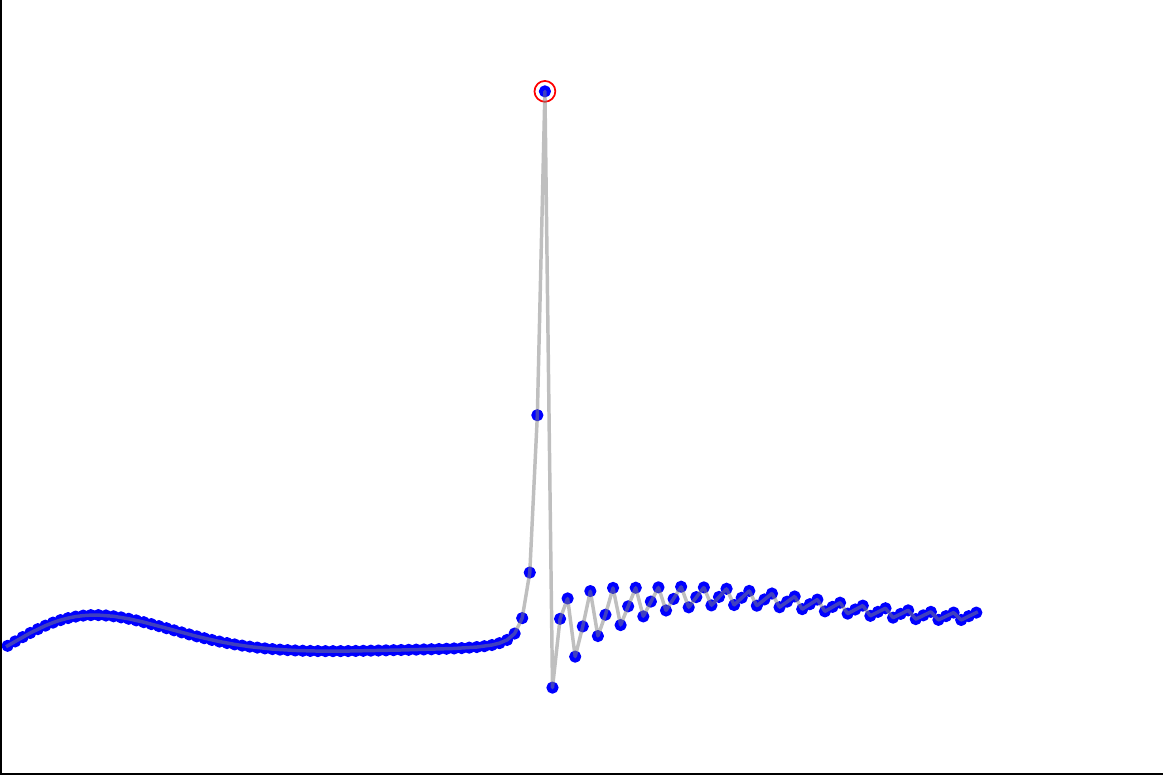};
\end{axis}
   \end{tikzpicture}
   \label{blowenergy}
   }\quad
   \subfloat[$z$ coordinate of a center particle versus time]{
   \begin{tikzpicture}[scale=.7]
      \begin{axis}
         [%
         axis on top,
         grid=major,
         enlargelimits=false, 
         xlabel={Time},
         ylabel={$z$},
         ytick align=outside,
         ]
         \addplot
         graphics
         [xmin=0.0,ymin=-1,xmax=.6,ymax=1]
         {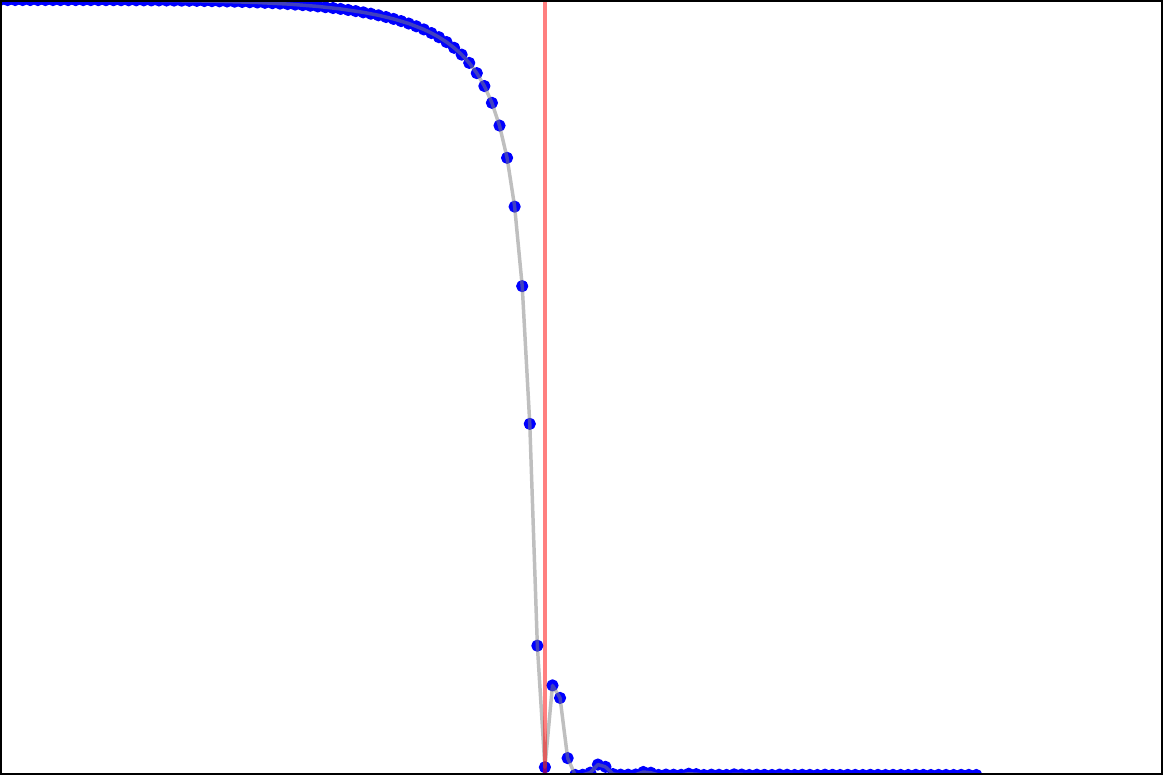};
\end{axis}
   \end{tikzpicture}
   \label{blowheight}
   }
   \caption{Blow-up of smooth initial data (\ref{busid}).}
   \label{fig:blowup_time}
\end{figure}

\begin{figure}
   \centering
      \input{coolwarmcolorbar}
   \begin{tikzpicture}
[spy using outlines={circle, magnification=2, connect spies},]
      \begin{axis}
         [%
         axis on top,
         enlargelimits=false, 
         xlabel={Radius},
         ylabel={Time},
         ytick align=outside,
         ytick={0,0.1,.2,0.28,.4,.5},
         colorbar=true,
         point meta min=-1,
         point meta max=1,
         ]
         \addplot
         graphics
         [xmin=0.0,ymin=0.0,xmax=.5,ymax=.5]
         {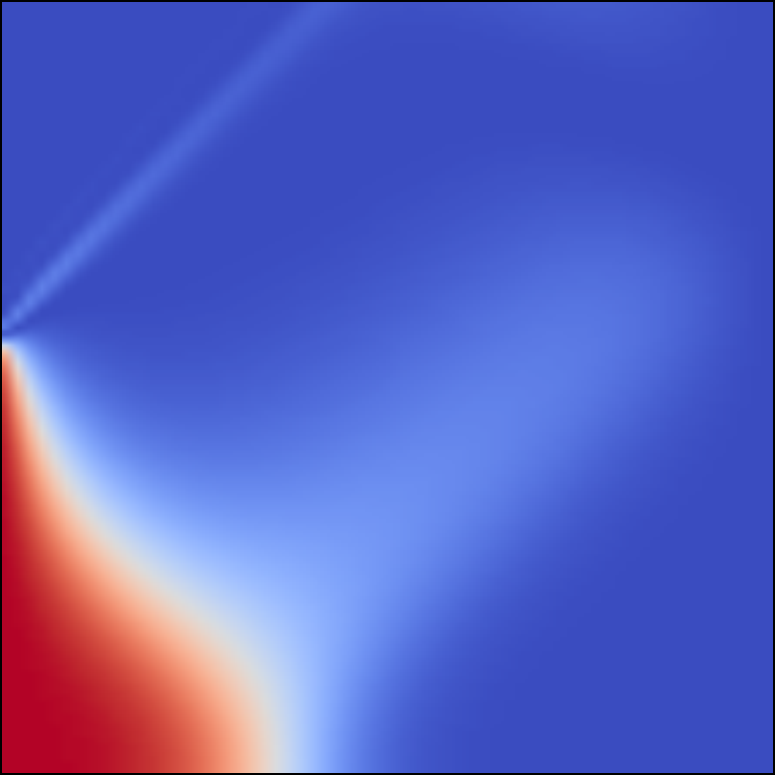};
\coordinate (spypoint) at (axis cs:.015,0.285);
\coordinate (magnifyingglass) at (axis cs:0.25,0.4);
\end{axis}
\spy[white, size=1.9cm] on (spypoint) in node[fill=white] at (magnifyingglass);
   \end{tikzpicture}
   \caption{Blow-up: view of the $z$ coordinate over time, versus a radius.
   The blow-up time, represented by the horizontal line, is $0.28$, confirming the experimental measurements from the literature.
}
   \label{fig:blowup_radius}
\end{figure}

The blow-up time measured in Figure\autoref{blowenergy} at $0.28$, as well as the flip observed in Figure\autoref{blowheight} are identical to the ones observed in~\cite{bfp07,kw13}.

%%%%%%%%%%%%%%%%%%%%%%%%%%%%%%%%%%%%%%%%%%%%%%%%%%%%%%%%%%%%%%%%%%%%%%%%%%%%%%%%%%%%
\subsection{Wave map equations with smooth potential}
\label{sec:potential}
Finally, we consider the discretisation of a wave map equation with an additional smooth potential
\begin{align}
   \label{eq:axispotential}
   V(u) = 400 \paren{u_1^2 + u_2^2}
   .
\end{align}
The setting is otherwise the same as in \autoref{sec:breather}.
We plot some snapshots of the numerical solution given by our multi-symplectic 
integrator in \autoref{fig:gordon}. The initial condition is a single winding 
around a great circle of the sphere, tilted from the equator plane at an angle of $45$ degrees. 
It is thus a fixed point of the wave map without potential. That initial condition is not a fixed point 
of the wave map with potential, as is evidenced by the snapshots in \autoref{fig:gordon}.

\newcommand*\subfloatgordon[2]{%
\subfloat[#2]{\includegraphics[width=.2\textwidth]{gordon/_gordon_#1}}
}

\begin{figure}
   \centering
\subfloatgordon{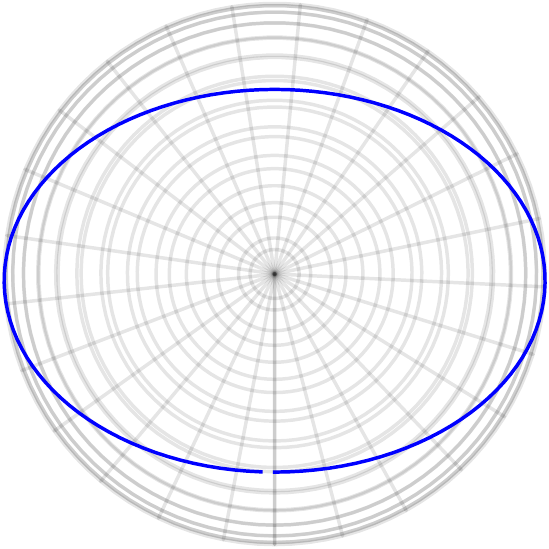}{0.000}\quad
\subfloatgordon{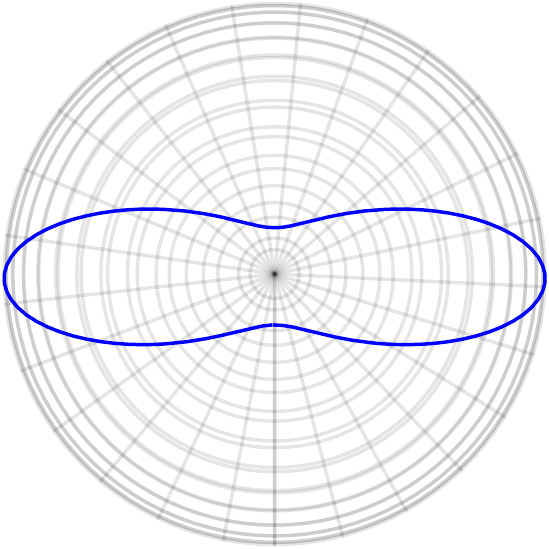}{0.082}\quad
\subfloatgordon{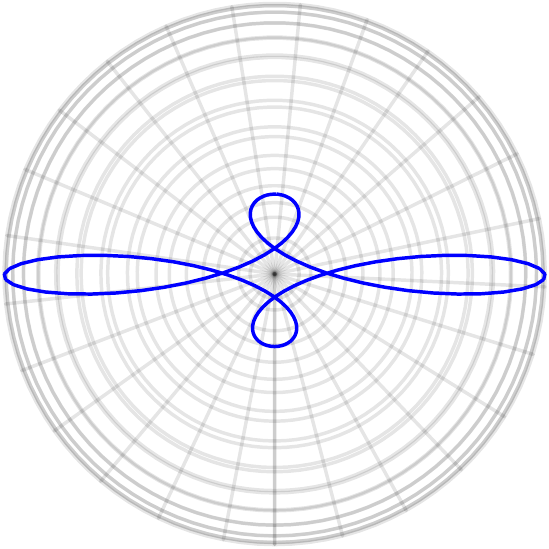}{0.117}\quad
\subfloatgordon{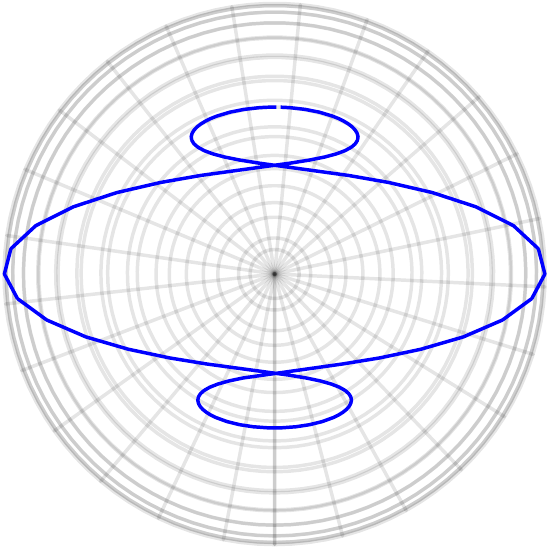}{0.195}\\
\subfloatgordon{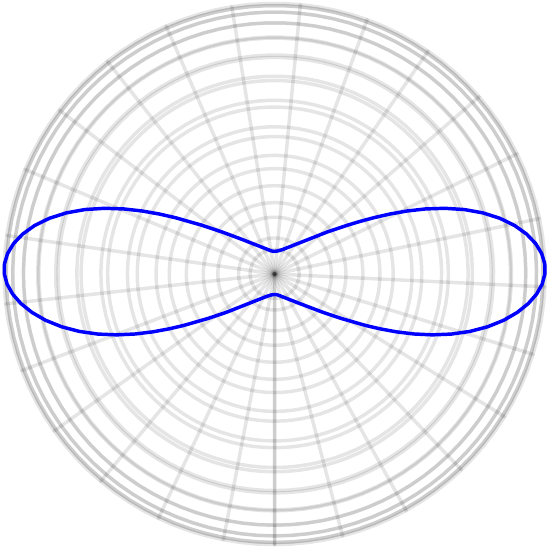}{0.258}\quad
\subfloatgordon{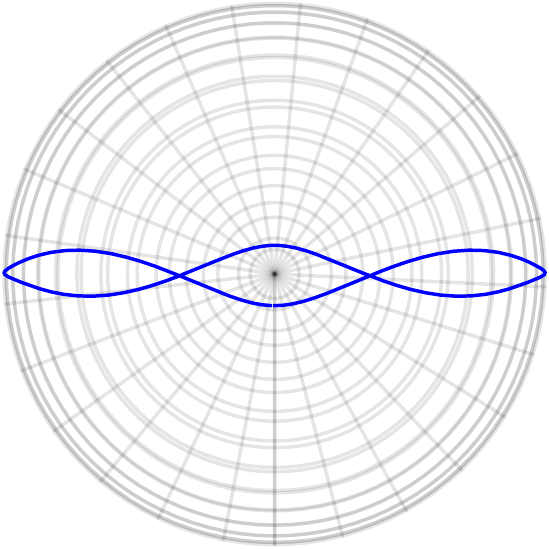}{0.273}\quad
\subfloatgordon{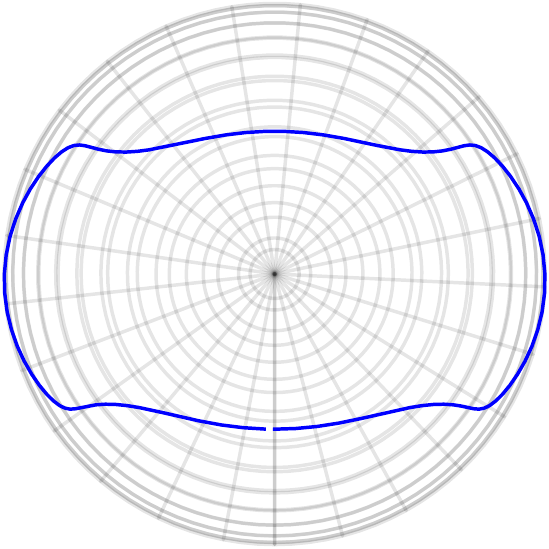}{0.328}\quad
\subfloatgordon{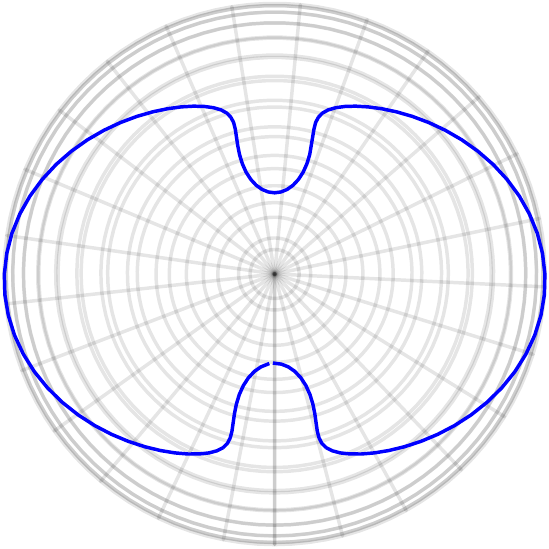}{0.410}\\
   \caption[]{
      Wave map with potential: snapshots of the wave map equation with the potential (\ref{eq:axispotential}).
      The caption below each snapshot is the time.
          }
   \label{fig:gordon}
\end{figure}

%%%%%%%%%%%%%%%%%%%%%%%%%%%%%%%%%%%%%%%%%%%%%%%%%%%%%%%%%%%%%%%%%%%%%%%%%%%%%%%%%%%%
\subsection{\olol{Wave map equations on a hyperbolic space}}

We use the hyperboloid model of the two-dimensional hyperbolic space.
The ambient space is $\R^3$ with the bilinear form of signature $(-,-,+)$, i.e.,
\begin{equation}
   \norm{(x,y,z)} \coloneqq -x^2 - y^2 + z^2
   .
\end{equation}

We represent the solutions of the wave map with the \emph{Poincaré disk} as a target.
Recall that the Poincaré disk is a stereographic projection of the hyperboloid on a disk or radius one.
A point of coordinates $(x,y,z)$ is projected to the point $(x/(1+z), y/1+z)$.

Our method works exactly in the same way, and the equation $\norm{u}$ now represents the ``sphere'' associated with that bilinear product, that is a two-sheet hyperboloid.
We always stay on the hyperboloid sheet with positive $z$-coordinate.

In \autoref{fig:hyperbolic}, we show the evolution of a wave map on the Poincaré disk, with initial condition, on the Poincaré disk identified as a subset of $\CC$, given by 
\begin{equation}
   z_0(\theta) = \e^{\ii \theta} + 0.3 \e^{\ii 8\theta} + 0.2 \e^{\ii 4 \theta}
   \qquad \theta \in [0,2\pi]
   .
\end{equation}

Note that the simulation does not take place directly on the Poincaré disk, but on the upper hyperbolic sheet, as it is a Riemannian submanifold of $\R^3$.
We also show an energy plot on \autoref{fig:hyperbolic_energy} which confirms that there is no energy drift along the numerical solution.

\newcommand*\subfloathyperbolic[2]{%
\subfloat[#2]{\includegraphics[width=.2\textwidth]{hyperbolic/hyperbolic_#1-crop}}
}

\begin{figure}
   \centering

\subfloathyperbolic{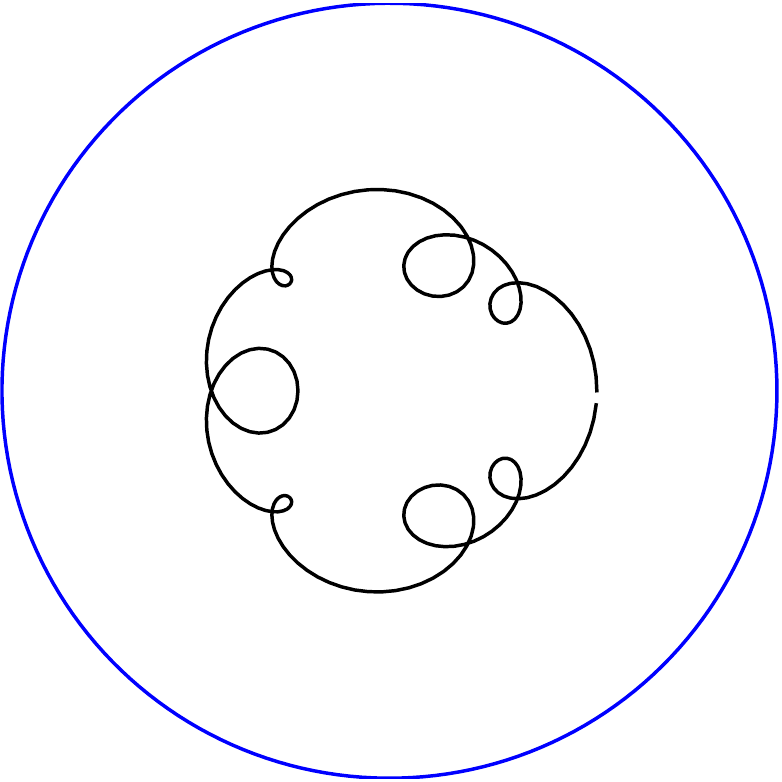}{0.000}\quad
\subfloathyperbolic{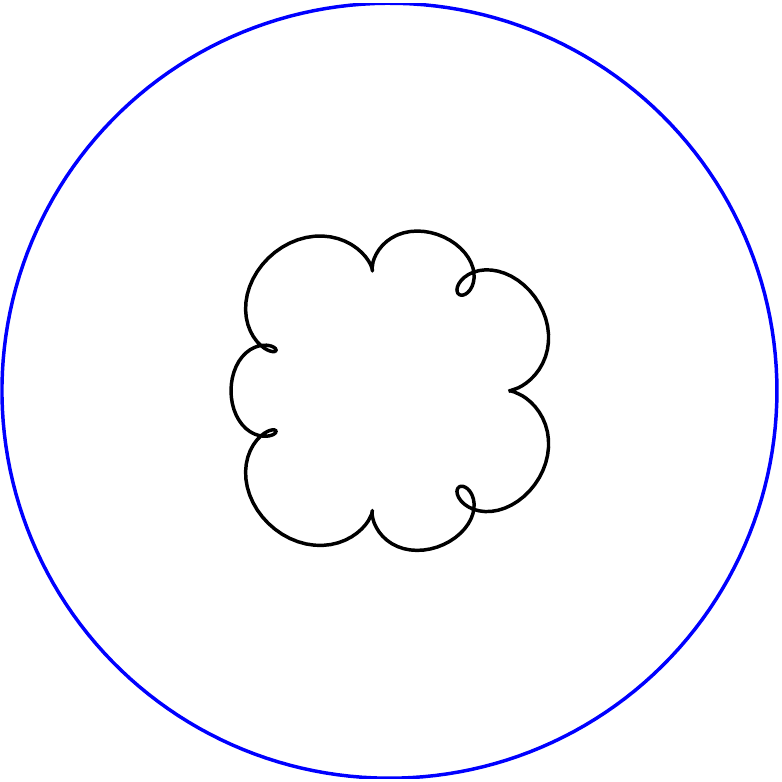}{0.041}\quad
\subfloathyperbolic{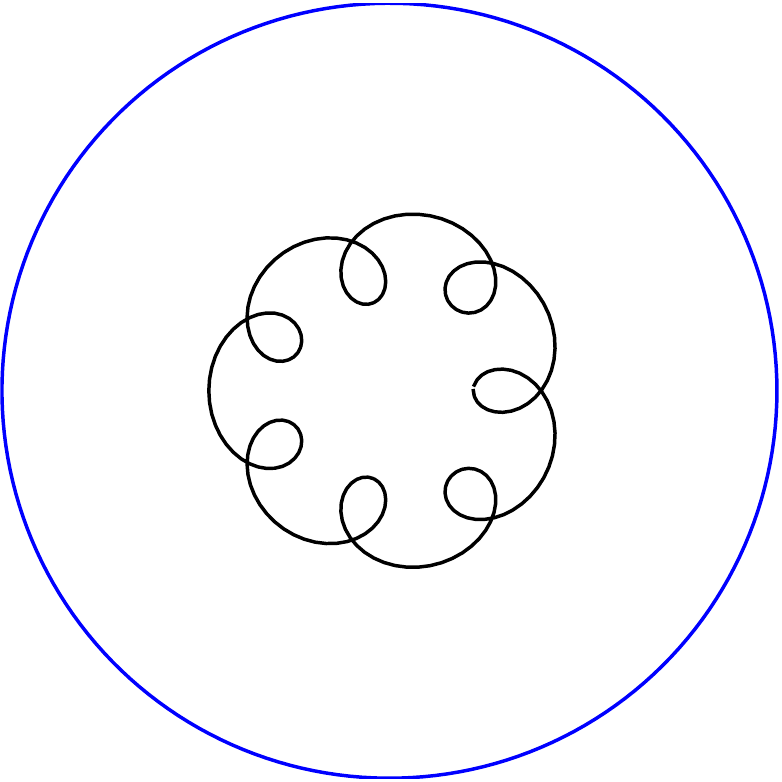}{0.059}\quad
\subfloathyperbolic{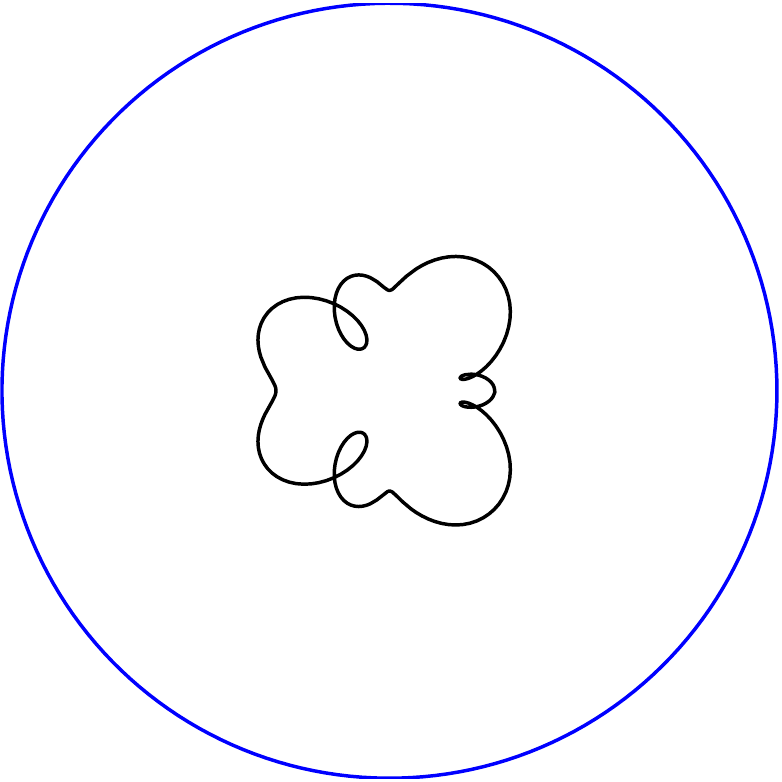}{0.098}\\
\subfloathyperbolic{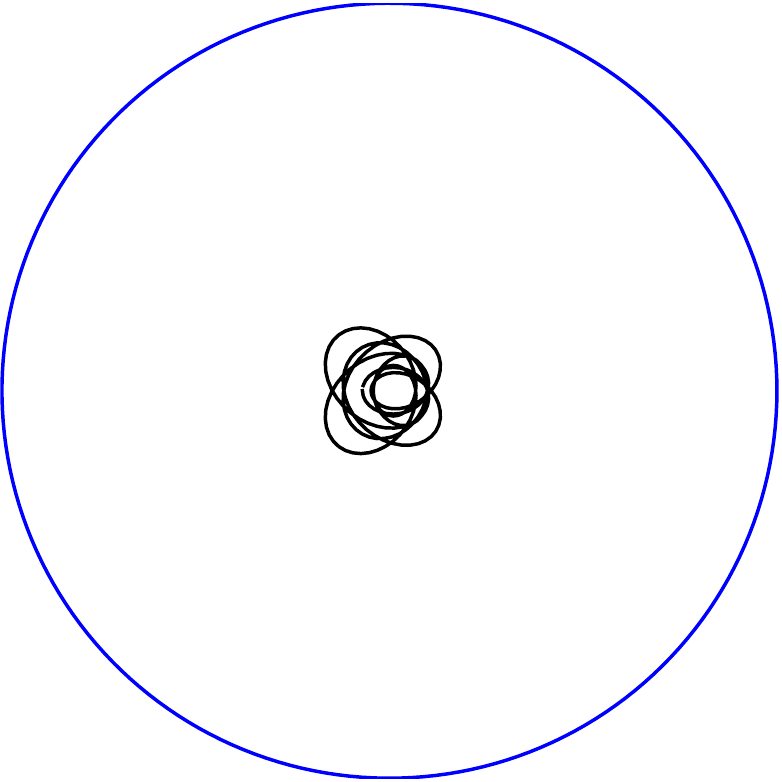}{0.195}\quad
\subfloathyperbolic{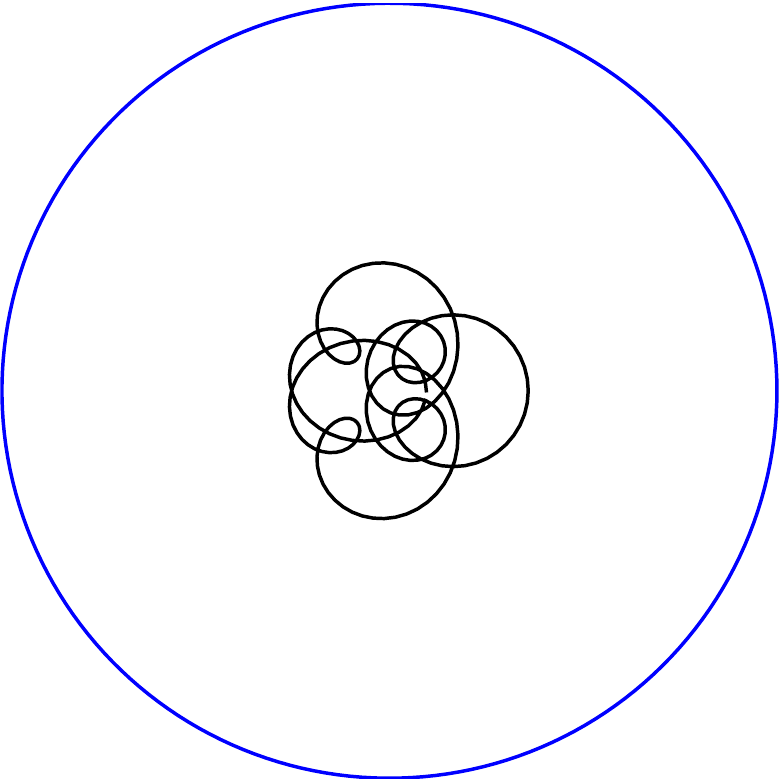}{0.234}\quad
\subfloathyperbolic{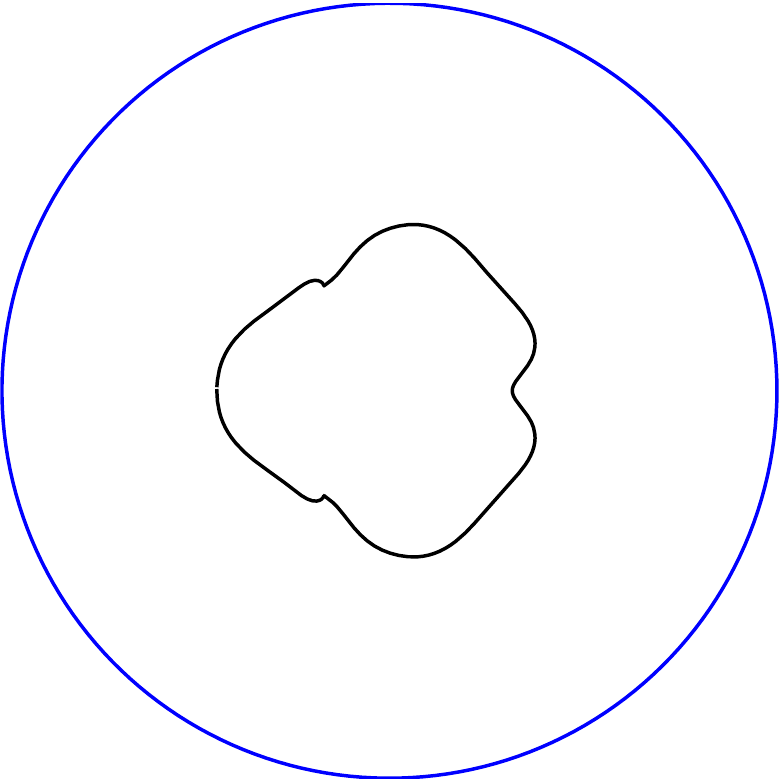}{0.332}\quad
\subfloathyperbolic{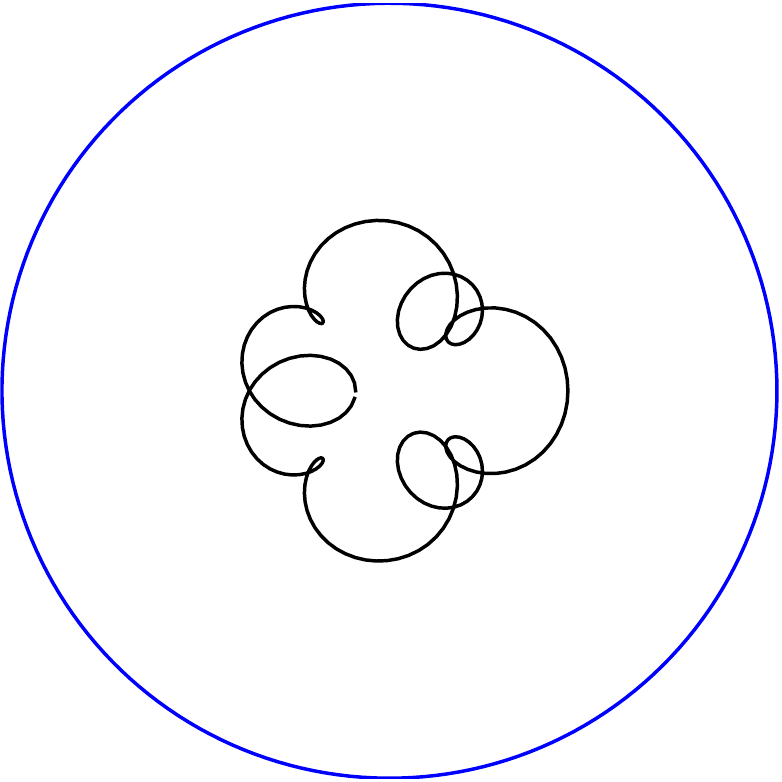}{0.469}\\

\caption[]{
      Wave map on a Poincaré disk.
      The caption below each snapshot is the time.
      In particular, the first plot represents the initial condition.
      There are $N = 2^8$ points, so the space step is $\Delta x = 1/N$, and we choose the time step $\Delta t = 0.5 \Delta x$.
      Note that the numerical solution is surprisingly stable, developing no chaotic behaviour.
      As this seems independent of the chosen initial condition, this could indicate some general integrability property of the wave map with a hyperbolic plane target.
          }
\label{fig:hyperbolic}
\end{figure}

\begin{figure}
\centering
\begin{tikzpicture}
      \begin{axis}
         [%
         axis on top,
         grid=both,
         enlargelimits=false, 
         ylabel={$(E-E_0)/E_0$},
         xlabel={Time},
         ytick align=outside,
         xtick align=outside,
         % ytick={0,0.1,.2,0.28,.4,.5},
         % colorbar=true,
         % point meta min=-1,
         % point meta max=1,
         ]
         \addplot
         graphics
         [xmin=0,ymin=-0.04,xmax=12,ymax=0.05]
         {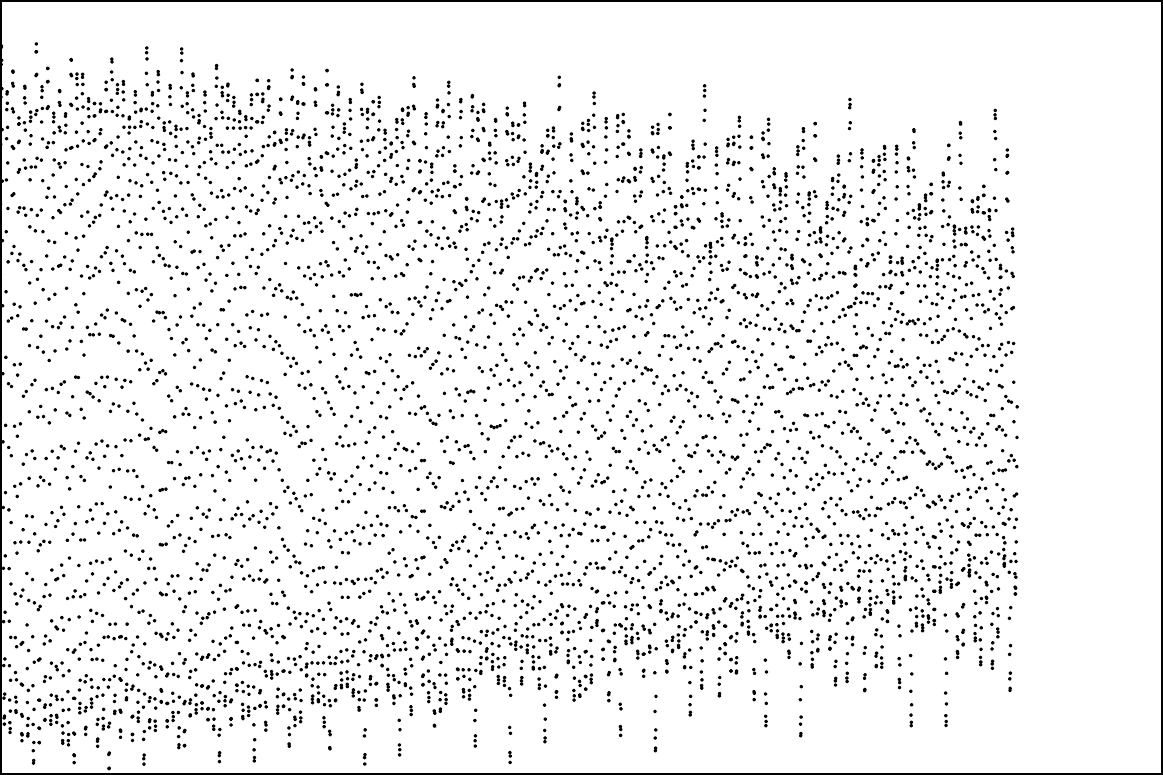};
\end{axis}
\end{tikzpicture}

\caption{Relative energy plot of the simulation of \autoref{fig:hyperbolic}.
The initial energy is $E_0 = -123$.
}
\label{fig:hyperbolic_energy}
\end{figure}

%%%%%%%%%%%%%%%%%%%%%%%%%%%%%%%%%%%%%%%%%%%%%%%%%%%%%%%%%%%%%%%%%%%%%
\section{Conclusion and open problems}\label{sect-conc}
In this paper, we have proposed and studied a new multi-symplectic 
numerical integrators for wave map equations on the sphere. 
This numerical scheme is explicit, conserves the constraint, 
has good conservation properties and can be seen as a generalisation 
of the \Shake algorithm for constrained mechanical systems. 
Furthermore, we observe convergence of order $2$ for smooth solutions. 

% So far, our numerical study confines to wave map equations on the sphere. 
\olol{Our method allows to treat wave map equations with other target manifolds which are submanifolds of $\R^n$, see also \autoref{apCPS}.}
Such examples could include classical Lie groups and symmetric spaces, see for instance~\cite{TeUh04}. 
But these are nontrivial extensions that may be the subject of future investigations.
Furthermore, it would also be interesting to understand whether different splitting of the multi-symplectic matrices 
could have some effect on the numerical discretisation.
In addition, it remains to develop, to try out and further analyse other classical multi-symplectic schemes 
such as the Preissman box scheme or some multi-symplectic Runge--Kutta collocation methods.
\olol{Such generalisations for constrained Hamiltonian PDEs seem far from straightforward.}

For all these reasons, it seems to us that it would be of interest to get 
more insight into the behaviour of multi-symplectic schemes for Hamiltonian PDEs 
with constraints as derived in this publication. 

%%%%%%%%%%%%%%%%%%%%%%%%%%%%%%%%%%%%%%%%%%%%%%%%%%%%%%%%%
\appendix
\section{\olol{Simulations on a complex projective space}}
\label{apCPS}

We explain how to use our method to simulate wave maps with target given by the complex projective space.
One possible application is to simulate the breathers described in~\cite[Example~8.2]{TeUh04}.

Our method needs the target manifold to be embedded as a submanifold of a Euclidean or Minkowski space.
In this case, we use the embedding of $\CP^n$ in $\uu{n+1}$, the space of complex-symmetric (Hermitian) square matrices of size $n+1$.
That space is equipped with the Frobenius scalar product $(\rho_1,\rho_2)\coloneqq \Tr(\rho_1 \rho_2)$. 

The complex projective space is defined as the submanifold
\begin{equation}
\label{eq:defCPn}
   \CP^n \coloneqq \setc{\rho \in \uu{n+1}}{\rho^2 = \rho \quad \Tr(\rho) = 1}
   .
\end{equation}
The standard definition of $\CP^n$ is by quotienting a vector $\Psi \in \CC^{n+1}$ by the equivalence relation $\Psi_1 \simeq \Psi_2 \coloneqq \bracket{\exists \lambda \in \CC \quad \Psi_2 = \lambda\Psi_1}$.
The map sending the standard representation of $\CP^n$ to the one described above in \eqref{eq:defCPn} is simply $\Psi \to  \Psi \Psi^*$, where we identify $\Psi$ with a $(n+1)\times 1$ complex matrix (``column vector''),
and where $\Psi^*$ denotes the conjugate transpose of the matrix $\Psi$.

The only issue is that of the projection on $\CP^n$.
Note that it is only a practical issue, as this setting already fits our framework exactly.

The constraint function is now defined on $S(n+1)$, is given by
\begin{equation}
   g(\rho) = (\rho^2 - \rho, \Tr(\rho)- 1)
   ,
\end{equation}
and takes values in $\uu{n+1}\times \CC$.

We observe that the second constraint is linear, so it will be automatically fulfilled by our method.
In practice, it can be ignored entirely.

At the projection step, we want to project an element $\widetilde{\rho} \in \uu{n+1}$ onto $\CP^n$ 
along the direction $\GG(\rho_0)$, and we may assume that $\Tr\paren{\widetilde{\rho}} = \ii$, i.e., the second constraint is already fulfilled.
The result of the projection is $\rho\in \CP^n$.
By differentiating $g$, the relation between $\rho$, $\widetilde{\rho}$ and $\rho_0$ is thus
\begin{equation}
\label{eq:projected}
   \rho = \widetilde{\rho} + \rho_0 \Lambda + \Lambda \rho_0 
   ,
\end{equation}
where $\Lambda \in \uu{n+1}$ is an unknown matrix.
We thus see that we will have $(n+1)^2/2$ real unknowns.

We now impose the constraint, and this gives the quadratic equation in the matrix $\Lambda$:
\begin{equation}
\label{eq:proj_lambda}
   \widetilde{\rho}^2 - \widetilde{\rho} + \widetilde{\rho}\Lambda + \Lambda \widetilde{\rho} + \Lambda^2 - \Lambda = 0
   .
\end{equation}
Solving $\Lambda$ in \eqref{eq:proj_lambda}, and using it in \eqref{eq:projected} now gives the projected value $\rho\in\CP^n$.

\section*{Acknowledgements}
DC acknowledges support from UMIT Research Lab at Umeå University.
OV acknowledges support from the J.C.~Kempe memorial fund (grant no.~SMK-1238).

\bibliographystyle{plain}
\bibliography{bibmsc}

\end{document}